\documentclass[11pt]{amsart}

\usepackage{amsmath,amssymb,latexsym,soul,cite,mathrsfs}

\usepackage{color,enumitem,graphicx}
\usepackage[colorlinks=true,urlcolor=blue,
citecolor=red,linkcolor=blue,linktocpage,pdfpagelabels,
bookmarksnumbered,bookmarksopen]{hyperref}
\usepackage[english]{babel}
\usepackage{tikz-cd}
\usepackage[left=2.9cm,right=2.9cm,top=2.8cm,bottom=2.8cm]{geometry}
\usepackage[hyperpageref]{backref}
\usepackage[all]{xy}
\usepackage[colorinlistoftodos]{todonotes}
\makeatletter
\providecommand\@dotsep{5}
\def\listtodoname{List of Todos}
\def\listoftodos{\@starttoc{tdo}\listtodoname}
\makeatother

\numberwithin{equation}{section}
\def\dis{\displaystyle}
\def\cal{\mathcal}

\newtheorem{theorem}{Theorem}[section]
\newtheorem{definition}{Definition}[section]
\newtheorem{proposition}[theorem]{Proposition}
\newtheorem{lemma}[theorem]{Lemma}
\newtheorem{corollary}[theorem]{Corollary}

\newtheorem{remark}{Remark}
\newtheorem{problem}{Problem}
\newcommand{\W}{\textrm W}

\newcommand\R{\mathbb R}
\newcommand\N{\mathbb N}

\begin{document}

\title[tempered fractional derivative]
{fractional tempered variational calculus}

\author{C\'esar E. Torres Ledesma}
\author{Gast\~ao S. F.  Frederico}
\author{Manuel M. Bonilla}
\author{Jes\'us A. Rodr\'iguez}

\address[C\'esar T. Ledesma]
{\newline\indent Instituto de Investigaci\'on en Matem\'aticas
\newline\indent
Faculta de Ciencias F\'isicas y Matem\'aticas
\newline\indent
Universidad Nacional de Trujillo
\newline\indent
Av. Juan Pablo II s/n. Trujillo-Per\'u}
\email{\href{ctl\_576@yahoo.es}{ctl\_576@yahoo.es}}

\address[Gast\~ao S. F. Frederico]
{\newline\indent Campus de Russas
\newline\indent
Universidade Federal de Cear\'a,
\newline\indent
62900-000, Russas - Brazil
\newline\indent
and
\newline\indent
Faculty of Science and Technology
\newline\indent
University of Cape Verde
\newline\indent
Praia, Santiago, Cape Verde}
\email{\href{gastao.frederico@ua.pt }{gastao.frederico@ua.pt }}

\address[Manuel M. Bonilla, Jes\'us A. Rodr\'iguez]
{\newline\indent Departamento de Matem\'aticas
\newline\indent
Universidad Nacional de Trujillo
\newline\indent
Av. Juan Pablo II s/n. Trujillo-Per\'u}
\email{\href{mailto: mmontalvo@unitru.edu.pe, javalos@unitru.edu.pe}{ mmontalvo@unitru.edu.pe, javalos@unitru.edu.pe}}

\pretolerance10000


\begin{abstract}
In this paper, we derive sufficient conditions ensuring the existence of a weak solution $u$ for a tempered fractional Euler-Lagrange equations 
$$
\frac{\partial L}{\partial x}(u,{^C}\mathbb{D}_{a^+}^{\alpha, \sigma} u, t) + \mathbb{D}_{b^-}^{\alpha, \sigma}\left(\frac{\partial L}{\partial y}(u, {^C}\mathbb{D}_{a^+}^{\alpha, \sigma}u, t) \right) = 0
$$
on a real interval $[a,b]$ and ${^C}\mathbb{D}_{a^+}^{\alpha, \sigma}, \mathbb{D}_{b^-}^{\alpha, \sigma}$ are the left and right Caputo and Riemann-Liouville tempered fractional derivatives respectively of order $\alpha$. Furthermore, we study a fractional tempered version of Noether theorem and we provide a very explicit expression of a constant of motion in terms of symmetry group and Lagrangian for fractional problems of calculus of variations. Finally we study a mountain pass type solution of the cited problem.
\end{abstract}



\maketitle

\section{Introduction}

Continuum-time random walk (CTRW) is a rigorous and general mathematical model for particle kinetics, which incorporates waiting times and/or non-Gaussian jump distributions with divergent second moments to account for the anomalous jumps called L\'evy flights \cite{metzler}. The continuous limit for such models leads to a fractional in time and/or space diffusion equation. However, in practice, many physical processes take place in bounded domains in finite times and have finite moments. Therefore, the divergent second moments may not be applicable to such processes. In order to overcome this modeling barrier, there are different techniques such as discarding the very large jumps and employing truncated L\'evy flights \cite{mantegna} or, adding a high-order power-law factor \cite{sokolov}. However, the most popular, and perhaps most rigorous, approach to get finite moments is exponentially tempering the probability of large jumps of L\'evy flights, which results in tempered-stable L\'evy processes with finite moments \cite{castillo, mm, Sabzikar}. The corresponding fluid (continuous) limit for such models yields the tempered fractional diffusion equation, which complements the previously known models in fractional calculus.

Nowadays tempered fractional calculus is considered a natural generalization of fractional calculus. The tempered fractional derivatives and integrals are obtained when the fractional derivatives and integrals are multiplied by an exponential factor \cite{liemert, sayed, mohsen}. 
Tempered fractional differential equations have been applied in different fields of physics such as in geophysics, statistical physics, plasma physics or in the context of astrophysics \cite{castillo, gajda, kulberg, mm, as, mohsen, zhang}. Apart from the physics field, the tempered fractional derivatives have also been applied in finance for modeling price fluctuations with semi-heavy tails \cite{Sabzikar}. 

The calculus of variations is one of the most traditional fields within mathematics and its applications. Although the idea of minimization as the basis of many physical laws had its origin among the Greeks, it was not until the last years of the 17th century that, with the help of the recently formalized differential calculus, the problem of finding the minimum of a function defined on a set of functions. Starting with Johan Bernoulli's public challenge, in 1696, to the best mathematicians of his time, to solve the problem of the curve giving the minimum travel time between two points on the plane, most of the great names of mathematics prior to this century it has been associated with the calculus of variations. Variational principles, starting with the principle of least action and the Hamilton-Jacobi theory, have been the driving force behind the conceptual development of physics and continue to contribute powerfully to the mathematical modeling of many other fields of science. Direct variational methods, pioneered by Hilbert, have allowed the spectacular growth of the study of partial differential equations and a good part of nonlinear analysis.

The calculus of variations and the fractional calculus are connected since the nineteenth century. Indeed, in 1823 Niels Heinrik Abel applied the fractional calculus to the solution of an integral equation that arises in the formulation of the tautochrone problem. This problem, sometimes also called the isochrone problem, consists to find the shape of a frictionless wire lying in a vertical plane such that the time of a bead placed on the wire slides to the lowest point of the wire in the same time regardless of where the bead is placed. It turns out that the cycloid is the isochrone as well as the brachistochrone curve, solving simultaneously the brachistochrone problem of the calculus of variations and Abel?s fractional problem (Abel 1965). It was however, only in the twentieth century that both areas joined in a unique research field: the fractional calculus of variations.

The fractional calculus of variations consists in extremizing (minimizing or maximizing) functionals whose Lagrangians contain fractional integrals and/or derivatives. It was born in 1996-1997 \cite{CD:Riewe:1996,CD:Riewe:1997}, when Riewe derived Euler-Lagrange fractional differential equations and showed how nonconservative systems in mechanics can be described using fractional derivatives. It is a remarkable result since frictional and nonconservative forces are beyond the usual macroscopic variational treatment and, consequently, beyond the most advanced methods of classical mechanics. For more details see \cite{dtorres1, dtorres2}.

Recently, Almeida and Morgado \cite{almeida1} studied variational problems where the cost functional involves the tempered Caputo fractional derivative. More precisely, they considered the following functional
\begin{equation}\label{A01}
\begin{aligned}
\mathcal{J}: C^1[a,b] &\to \R\\
x &\to \mathcal{J}(x) = \int_{a}^{b} L(t, x(t), {^{C}}\mathbb{D}_{a^+}^{\alpha, \sigma}x(t))dt,
\end{aligned}
\end{equation}
where $\alpha \in (0,1)$, $\sigma \geq 0$ and $L:[a,b]\times \R^2 \to \R$  is continuously differentiable with respect to the second and third variables. They showed that if $x^*$ is a local minimizer of $\mathcal{J}$, subject to the boundary conditions
$$
x(a) = \beta_1\quad \mbox{and}\quad x(b) = \beta_2,
$$
where $\beta_1, \beta_2\in \R$, then $x^*$ is a solution of the Euler-Lagrange equation
\begin{equation}\label{A02}
\partial_2 L[x](t) + \mathbb{D}_{b^-}^{\alpha, \sigma}(\partial_3 L[x](t)) = 0,\quad t\in [a,b],
\end{equation}
where $L[x](t) = L(t, x(t), \mathbb{D}_{a^+}^{\alpha, \sigma}x(t))$. In the case $\sigma = 0$, (\ref{A01}) reduce to the functional with Caputo fractional derivative, that is 
$$
\begin{aligned}
\mathcal{L}(x) = \int_{a}^{b} L(t, x(t), {^{C}}D_{a^+}^{\alpha}x(t))dt.
\end{aligned}
$$   
By introducing the functional space $E_{\alpha, p}$, which is defined as 
$$
E_{\alpha, p}:= \{u\in L^p\;\;\mbox{satisfying $D_{-}^{\alpha}u\in L^p$ and $I_{-}^{\alpha}\circ D_{-}^{\alpha} u = u$ a.e.}\}
$$
Bourdin \cite{bourdin} derived sufficient conditions to ensure the existence of weak solution $u$ for the following fractional Euler-Lagrange equation
$$
\frac{\partial L}{\partial x} (u, D_{-}^{\alpha}u, t) + D_{+}^{\alpha} \left(\frac{\partial L}{\partial y} (u, D_{-}^{\alpha}u, t) \right) =0,
$$ 
which is a critical point of the functional $\mathcal{L}$.

Motivated by these previous works, in the present paper  we study sufficient conditions to ensure the existence of critical points of the following Lagrangian functional 
$$
\mathcal{J}(u) = \int_{a}^{b}L(u, {^C}\mathbb{D}_{a^+}^{\alpha, \sigma}u, t)dt,
$$
where 
$$
\begin{aligned}
L: \R\times \R \times [a,b] &\to \R\\
(x,y,t)&\to L(x,y,t),
\end{aligned}
$$
$\alpha \in (\frac{1}{2},1)$, $\sigma >0$, $a<b$, ${^C}\mathbb{D}_{a^+}^{\alpha, \sigma}$ is the left Caputo tempered fractional derivative and where the variable $u$ is a function defined almost everywhere on $(a,b)$ with values in $\R$. It is well-known that critical points of the functional $\mathcal{J}$ are characterized by the solutions of the fractional Euler-Lagrange equation
\begin{equation}\label{EL}
\frac{\partial L}{\partial x}(u,{^C}\mathbb{D}_{a^+}^{\alpha, \sigma} u, t) + \mathbb{D}_{b^-}^{\alpha, \sigma}\left(\frac{\partial L}{\partial y}(u, {^C}\mathbb{D}_{a^+}^{\alpha, \sigma}u, t) \right) = 0,
\end{equation}
where $\mathbb{D}_{b^-}^{\alpha, \sigma}$ is the right Riemann-Liouville tempered fractional derivative. 

The concept of symmetry plays an important role in science and engineering. Symmetries are described by transformations, which result in the same object after the transformations are carried out. They are described mathematically by parameter groups of transformations \cite{PO}.

Noether’s principle establishes a relation between the existence of symmetries and the
existence of conservation laws, and is one of the most beautiful results of the calculus of variations
and mechanics \cite{FT,Noether:1918}.

Fractional version of Noether’s theorem is the subject of the present work and we provide a very explicit expression of a constant of motion in terms of symmetry group and Lagrangian for fractional problems of calculus of variations. Our main result provides significant novelty, because Several authors have already proved that kind of result, but in our opinion, none of them obtained such an explicit expression of a fractional constant of motion \cite{Atana, bourdin1,FL}. For more details see Theorem~\ref{CNT} bellow.

\section{Some previous results}

For $\alpha >0$ and $x\geq 0$, the incomplete Gamma function is defined by
$$
\gamma (\alpha, x) = \int_{0}^{x} t^{\alpha-1}e^{-t}dt,
$$
which is convergen for all $\alpha >0$. This function has the following estimates
\begin{equation}\label{g01}
e^{-x}\frac{x^\alpha}{\alpha} \leq \gamma (\alpha, x) \leq
\frac{x^{\alpha}}{\alpha},
\end{equation}
and by using integration by parts, we see that
\begin{equation}\label{g02}
\gamma (\alpha+1, x) = \alpha \gamma (\alpha, x) - x^{\alpha}e^{-x}.
\end{equation}
This equality can be used to extend the definition of $\gamma (\alpha, x)$ to negative, non integer values of $\alpha$. For more  details the reader's can see \cite{fisher}.

For $\alpha \in (0,1)$ and $\sigma > 0$, the left and right Riemann-Liouville tempered fractional integrals of order $\alpha$ are defined as
\begin{equation}\label{T01}
\mathbb{I}_{a^+}^{\alpha, \sigma}u(x) = \frac{1}{\Gamma (\alpha)}\int_{a}^{x} (x-s)^{\alpha-1}e^{-\sigma(x-s)}u(s)ds,\;\;x>a,
\end{equation}
and
\begin{equation}\label{T02}
\mathbb{I}_{b^-}^{\alpha, \sigma}u(x) = \frac{1}{\Gamma (\alpha)} \int_{x}^{b} (s-x)^{\alpha-1}e^{-\sigma(s-x)}u(s)ds,\;\;x< b,
\end{equation}
respectively. For $\alpha=0$, we define $\mathbb{I}_{a^+}^{0, 0}u(x)=\mathbb{I}_{b^-}^{0, 0}u(x) =u(x)$.

Moreover, for $\alpha \in (0,1)$ and $\sigma >0$, the left and right Riemann-Liouville tempered fractional derivatives of order $\alpha$ are defined as
\begin{equation}\label{T03}
\mathbb{D}_{a^+}^{\alpha,\sigma}u(x) = (e^{-\sigma x}{_{a}}D_{x}^{\alpha} e^{\sigma x})u(x) = \frac{e^{-\sigma x}}
{\Gamma (1-\alpha)} \frac{d}{dx}\int_{a}^{x} (x-s)^{-\alpha}e^{\sigma s}u(s)ds,\;\;x>a,
\end{equation}
and
\begin{equation}\label{T04}
\mathbb{D}_{b^-}^{\alpha; \sigma}u(x) = (e^{\sigma x}{_{x}}D_{b}^{\alpha} e^{-\sigma x})u(x) =
 \frac{-e^{\sigma x}}{\Gamma (1-\alpha)} \frac{d}{dx} \int_{x}^{b} (s-x)^{-\alpha}e^{-\sigma s}u(s)ds,\;\;x <
 b,
\end{equation}
respectively. An alternative approach in defining the tempered fractional derivatives is based on the left-sided and right-sided tempered Caputo fractional derivatives of order $\alpha$, defined, respectively, as
\begin{equation}\label{T05}
{^{C}}\mathbb{D}_{a^+}^{\alpha,\sigma} u(x) = \left( e^{-\sigma x}
{^C_{a}}D_{x}^{\alpha} e^{\sigma x}\right)u(x) = \frac{e^{-\sigma
x}}{\Gamma (1-\alpha)}\int_{a}^{x}(x-s)^{-\alpha} [e^{\sigma
s}u(s)]'ds,\;\;x>a,
\end{equation}
and
\begin{equation}\label{T06}
{^{C}}\mathbb{D}_{b^-}^{\alpha,\sigma}u(x) = \left(e^{\sigma x} {^C_x}D_{b}^{\alpha} e^{-\sigma x} \right)u(x) = \frac{-e^{\sigma x}}{\Gamma (1-\alpha)}\int_{x}^{b}(s-x)^{-\alpha} [e^{-\sigma s}u(s)]'ds,\;\;x < b.
\end{equation}

This operators have the following properties.
\begin{lemma}\label{lm01}
\cite{torres3} Let $\alpha >0$, $\sigma > 0$ and $u\in AC[a,b]$. Then $\mathbb{I}_{a^+}^{\alpha, \sigma} u, \mathbb{I}_{b^-}^{\alpha, \sigma}u$ are well defined. Moreover
\begin{equation}\label{T07}
\mathbb{I}_{a^+}^{\alpha, \sigma} u(x) = \frac{u(a)}{\sigma^\alpha
\Gamma (\alpha)} \gamma (\alpha, \sigma (x-a)) +
\frac{1}{\sigma^\alpha \Gamma (\alpha)} \int_{a}^{x} \gamma
(\alpha, \sigma (x-t))u'(t)dt,
\end{equation}
and
\begin{equation}\label{T08}
\mathbb{I}_{b^-}^{\alpha, \sigma} u(x) = \frac{u(b)}{\sigma^\alpha \Gamma (\alpha)} \gamma (\alpha, \sigma (b-x)) - \frac{1}{\sigma^\alpha \Gamma (\alpha)}\int_{x}^{b} \gamma (\alpha, \sigma (t-x))u'(t)dt.
\end{equation}
\end{lemma}

\begin{theorem}\label{Lp}
\cite{torres3} Let $\alpha \in (0,1)$, $\sigma >0 $, $p\in [1, \infty]$. Then, the tempered fractional integrals of Riemann-Liouville $\mathbb{I}_{a^+}^{\alpha, \sigma}, \mathbb{I}_{b^-}^{\alpha, \sigma}: L^p(a,b) \to L^p(a,b)$ are bounded. Moreover
\begin{equation}\label{09}
\|\mathbb{I}_{a^+}^{\alpha, \sigma}u\|_{L^p(a,b)} \leq
\frac{\gamma (\alpha, \sigma (b-a))}{\sigma^\alpha \Gamma
(\alpha)} \|u\|_{L^p(a,b)},
\end{equation}
and
\begin{equation}\label{10}
\|\mathbb{I}_{b^-}^{\alpha, \sigma}u\|_{L^p(a,b)} \leq
\frac{\gamma (\alpha, \sigma (b-a))}{\sigma^\alpha \Gamma
(\alpha)}\|u\|_{L^p(a,b)}.
\end{equation}
\end{theorem}

\begin{theorem}\label{teo2-1}
\cite{torres3} Let $p>1$, $\alpha \in \left(0, \frac{1}{p} \right)$, $\sigma >0$ and $p_{\alpha}^{*} = \frac{p}{1-\alpha p}$. Then, the Riemann-Liouville tempered fractional integrals $\mathbb{I}_{a^+}^{\alpha, \sigma}, \mathbb{I}_{b^-}^{\alpha, \sigma}$ are bounded from $L^p(a,b)$ into $L^{p_{\alpha}^{*}}(a,b)$.
\end{theorem}

The following integration by parts theorem for Riemann-Liouville tempered fractional integrals holds.
\begin{theorem}\label{IPP}
Let $\alpha \in (0,1)$, $\sigma >0$, $p\in (1, \infty)$, $q\in (1, \infty)$ and
$$
\frac{1}{p} + \frac{1}{q} \leq 1+\alpha.
$$
If $u\in L^p(a,b)$ and $v\in L^q(a,b)$, then
\begin{equation}\label{pp01}
\int_{a}^{b} \mathbb{I}_{a^+}^{\alpha, \sigma}u(x) v(x)dx = \int_{a}^{b} u(x) \mathbb{I}_{b^-}^{\alpha, \sigma}v(x)dx.
\end{equation}
\end{theorem}
Now we consider some smoothness properties of the Riemann-Liouville tempered fractional integrals

\begin{theorem}\label{conti}
\cite{torres3} Let $\alpha \in (0,1)$, $\sigma >0$ and $u\in C[a,b]$. Then the tempered fractional integrals of Riemann-Liouville $\mathbb{I}_{a^+}^{\alpha, \sigma}u, \mathbb{I}_{b^-}^{\alpha, \sigma}$ are continuos on $[a,b]$ and
\begin{equation}\label{limit}
\lim_{x\to a^+} \mathbb{I}_{a^+}^{\alpha, \sigma}u(x) = 0\quad \mbox{and}\quad \lim_{x\to b^-} \mathbb{I}_{b^-}^{\alpha, \sigma}u(x) = 0.
\end{equation}
Moreover,
$$
\|\mathbb{I}_{a^+}^{\alpha, \sigma}u\|_\infty \leq
\frac{1}{\sigma^\alpha \Gamma (\alpha)} \gamma (\alpha, \sigma
(b-a))\|u\|_\infty,
$$
and
$$
\|\mathbb{I}_{b^-}^{\alpha, \sigma}u\|_\infty \leq
\frac{1}{\sigma^\alpha \Gamma (\alpha)} \gamma (\alpha, \sigma
(b-a))\|u\|_\infty.
$$
\end{theorem}

\begin{theorem}\label{HC}
\cite{torresT1} Let $\alpha \in (\frac{1}{2}, 1)$ and $\sigma >0$. Then, for each $u\in L^2(a,b)$, $\mathbb{I}_{a^+}^{\alpha, \sigma}u\in H^{\alpha-\frac{1}{2}}(a,b)$ and
$$
\lim_{x\to a^+} \mathbb{I}_{a^+}^{\alpha, \sigma}u(x) = 0,
$$
where $H^{\alpha-\frac{1}{2}}(a,b)$ denotes the H\"older space of order $\alpha-\frac{1}{2}>0$.
\end{theorem}

Considering the Riemann-Liouville and Caputo tempered fractional
derivative we have the following properties.
\begin{theorem}\label{teo01}
\cite{torresT1} Let $\alpha \in (0,1)$, $\sigma >0$ and $u\in AC[a,b]$. Then
\begin{equation}\label{T09}
\mathbb{D}_{a^+}^{\alpha, \sigma}u(x) =
\frac{u(a)}{\Gamma(1-\alpha)} (x-a)^{-\alpha}e^{-\sigma (x-a)}
+{^{C}}\mathbb{D}_{a^+}^{\alpha, \sigma}u(x),
\end{equation}
and
\begin{equation}\label{T10}
\mathbb{D}_{b^-}^{\alpha, \sigma} u(x) = \frac{u(b)}{\Gamma (1-\alpha)}(b-x)^{-\alpha}e^{-\sigma (b-x)} + {^{C}}\mathbb{D}_{b^-}^{\alpha, \sigma}u(x).
\end{equation}
\end{theorem}

\begin{theorem}\label{TFC}
\cite{torresT1} For $\alpha \in (0,1)$, $\sigma >0$ and $u\in AC[a,b]$, we have
\begin{enumerate}
\item
$$
\begin{aligned}
&{^C}\mathbb{D}_{a^+}^{\alpha, \sigma} \mathbb{I}_{a^+}^{\alpha, \sigma} u(x) = u(x),\\
&{^{C}}\mathbb{D}_{b^-}^{\alpha, \sigma} \mathbb{I}_{b^-}^{\alpha, \sigma}u(x) = u(x).
\end{aligned}
$$
\item
$$
\begin{aligned}
&\mathbb{I}_{a^+}^{\alpha, \sigma} {^{C}}\mathbb{D}_{a^+}^{\alpha, \sigma} u(x) = u(x) - e^{-\sigma (x-a)}u(a),\\
& \mathbb{I}_{b^-}^{\alpha, \sigma} {^{C}}\mathbb{D}_{b^-}^{\alpha, \sigma} u(x) = u(x) - e^{-\sigma (x-b)}u(b).
\end{aligned}
$$
\end{enumerate}
\end{theorem}

\noindent 
Finally we consider the integration by parts theorem for Riemann-Liouville tempered fractional derivative.
\begin{theorem}\label{IPPD}
\cite{torresT1} Let $\alpha \in (0,1)$, $\sigma >0$ and $u, v\in AC[a,b]$, then
\begin{equation}\label{T15}
\int_{a}^{b}u(x) \mathbb{D}_{b^-}^{\alpha, \sigma} v(x)dx =
\lim_{x\to a^+}u(x) \mathbb{I}_{b^-}^{1-\alpha, \sigma}v(x) -
\lim_{x\to b^-}u(x) \mathbb{I}_{b^-}^{1-\alpha, \sigma}v(x) +
\int_{a}^{b} {^{C}}\mathbb{D}_{a^+}^{\alpha, \sigma}u(x) v(x)dx.
\end{equation}
\end{theorem}

\section{Tempered fractional space of Sobolev type}

In this section we introduce the tempered fractional space of Sobolev type $\mathbb{H}_{0}^{\alpha, \sigma}(a,b)$ considered in \cite{torresT1}.

For $\alpha \in (0,1)$ and $\sigma >0$ define 
$$
\mathbb{H}_{0}^{\alpha, \sigma}(a,b) = \overline{C_{0}^{\infty}(a,b)}^{\|\cdot\|_{\alpha, \sigma}},
$$
where
\begin{equation}\label{B01}
\|u\|_{\alpha, \sigma} = \left( \int_{a}^{b}|u(x)|^2dx + \int_{a}^{b} |{^C}\mathbb{D}_{a^+}^{\alpha, \sigma}u(x)|^2dx \right)^{1/2},
\end{equation}
which is derived from the inner product
\begin{equation}\label{B-01}
\langle u,v\rangle_{\alpha, \sigma} = \int_{a}^bu\;vdx + \int_{a}^{b} {^C}\mathbb{D}_{a^+}^{\alpha, \sigma}u \; {^C}\mathbb{D}_{a^+}^{\alpha, \sigma} vdx.
\end{equation}
Furthermore $(\mathbb{H}_{0}^{\alpha, \sigma}(a,b), \langle \cdot, \cdot, \rangle_{\alpha, \sigma})$ is a Hilbert space.

Now we consider some basic properties of this function space which are proved in \cite{torresT1}.

\begin{lemma}\label{Blm01}
For any $u\in \mathbb{H}_{0}^{\alpha, \sigma}(a,b)$, we have
$$
\mathbb{I}_{a^+}^{\alpha, \sigma} {^C}\mathbb{D}_{a^+}^{\alpha, \sigma}u(x) = u(x),\;\;\mbox{a.e. in $(a,b)$}.
$$
\end{lemma}



\begin{theorem}\label{embedding}
Let $\alpha \in (\frac{1}{2}, 1)$ and $\sigma >0$, then $\mathbb{H}_{0}^{\alpha, \sigma}(a,b)$ is continuously embedded into $C(a,b)$. Moreover 
\begin{equation}\label{conti}
\|u\|_{\infty} \leq \frac{\sqrt{\gamma (2\alpha-1, 2\sigma (b-a))}}{(2\sigma)^{\alpha-\frac{1}{2}}\Gamma (\alpha)}\|u\|.
\end{equation}
\end{theorem}

\begin{theorem}\label{compact}
Let $\alpha \in (\frac{1}{2}, 1)$ and $\sigma >0$. Then the embedding
$$
\mathbb{H}_{0}^{\alpha, \sigma} (a,b) \hookrightarrow C\overline{(a,b)}
$$
is compact.
\end{theorem}
\begin{remark}
If $\alpha \in (\frac{1}{2}, 1)$ and $\sigma >0$, then for every $u\in \mathbb{H}_{0}^{\alpha, \sigma}(a,b)$, there exists $(\phi_n)_{n\in \N} \subset C_{0}^{\infty}(a,b)$ such that
$$
\lim_{n\to \infty}\|u-\phi_n\| = 0.
$$
Combining this limit with Theorem \ref{embedding} we arrive to
$$
0\leq |u(a)| = |u(a) - \phi_n(a)|\leq \frac{1}{(2\sigma)^{\alpha-\frac{1}{2}}\Gamma (\alpha)}[\gamma (2\alpha-1, 2\sigma (b-a))]^{1/2}\|u-\phi_n\| \to 0\;\;\mbox{as}\;\;n\to \infty.
$$
Consequently $u(a) = 0$. Similarly we can obtain that $u(b) = 0$.

By other side, Lemma \ref{Blm01} yields that
$$
{^{C}}\mathbb{D}_{a^+}^{\alpha, \sigma}u\in L^2(a,b).
$$
Therefore, $\mathbb{H}_{0}^{\alpha, \sigma}(a,b)$ can be rewritten as
$$
\mathbb{H}_{0}^{\alpha, \sigma}(a,b) = \{u\in L^2(a,b): \;{^{C}}\mathbb{D}_{a^+}^{\alpha, \sigma}u\in L^2(a,b)\;\;\mbox{and}\;\;u(a) = u(b) = 0\}.
$$
\end{remark}

\section{fractional tempered calculus of variations}

In this section, we prove for fractional tempered variational problem a existence theorem,  the fractional tempered Euler-Lagrange equation and the tempered fractional Noether’s theorem.

We start with some definition of basic notions and notations.

\begin{definition}
	Let $\alpha \in \left(\frac{1}{2}, 1\right)$ and $a,b \in \R$, with $a<b $. 	Let $\Omega$  be a connected open subset of $ \R$ and $ \overline{\Omega}$ the closure of $\Omega$, such that $\mathbb{H}_{0}^{\alpha, \sigma}(a,b)\subset \Omega$,
	  then
	\begin{itemize}
		
		\item[(i)] For $\alpha \in (\frac{1}{2}, 1)$ and $\sigma >0$ we define the Lagrangian functional $\mathcal{J}$ on $\mathbb{H}_{0}^{\alpha, \sigma}(a,b)$ as 
		\begin{equation}\label{EL1}
			\begin{aligned}
				\mathcal{J}: \mathbb{H}_{0}^{\alpha, \sigma}(a,b)&\longrightarrow \R\\
				u&\longmapsto \mathcal{J}(u) = \int_{a}^{b} L\left(u(t), {^{C}}\mathbb{D}_{a^+}^{\alpha, \sigma}u(t), t\right)\textnormal{d}t<\infty,
			\end{aligned}
		\end{equation}
		where the Lagrangian $  L :\overline{\Omega}\times\mathbb{R}\times[a,b]\longrightarrow \mathbb{R}$ is of class $\textnormal{C}^1$. The operator ${^{C}}\mathbb{D}_{a^+}^{\alpha, \sigma}$ denotes left-sided  tempered Caputo fractional derivatives of order $\alpha$.
		
		\item[(ii)]  Let $u,v \in  \textnormal{AC}\left([a,b],\R\right)$ such that $u+ hv\in  \overline{\Omega}$ and $\mathcal{J}(u+hv)\in \textnormal{L}^{ 1}\left( a,b \right)$ for $h$ in a neighborhood of $0$. Then,  \emph{the first-order Gâteaux-differential} of $\mathcal{J}$ at $u$ in direction $v$ is defined by
		\begin{equation}\label{EL2}
				v\longmapsto D\mathcal{J}(u)v := \lim_{h\to 0} \frac{\mathcal{J}(u+hv) - \mathcal{J}(u)}{h}
		\end{equation}
		provided that the right-hand side term exists. The function $u$ is called \emph{weak extremal} of $\mathcal{J}$ if $D\mathcal{J}(u)v=0$ for all $v\in \textnormal{AC}\left([a,b],\R\right)$.
		
		\item[(iii)] Function $u_0\in \overline{\Omega}$ is a \emph{minimizer (or global minimizer)} of the functional $\mathcal{J}$
		 if
		$$\mathcal{J}(u_0)\leq \mathcal{J}(u) $$ for all $u\in\overline{\Omega}$.
	\end{itemize}	
\end{definition}

We are concerned with  the following optimisation problem:	
\begin{problem}\label{pb}
	Find $u_0\in \overline{\Omega}$ such that $$\mathcal{J}\left(u_0\right)=\textnormal{min}\{\mathcal{J}(u)\,|\,u\in \overline{\Omega}\}.$$
\end{problem}

The next theorem is a Tonelli type theorem for Problem~\ref{pb}:
\begin{theorem}\label{main1}
Let $L$ be a Lagrangian of class $\textnormal{C}^1$ and $\alpha \in \left(\frac{1}{2}, 1\right)$. If $L$ satisfies the following conditions:
\begin{enumerate}
\item[$(L_1)$] There is $d_1 \in (0,2]$ and $r_1, s_1 \in \textnormal{C}(\R\times [a,b], \R^+)$ such that 
$$
\forall (x,y,t)\in \overline{\Omega}\times \R \times[a,b], \quad |L(x,y,t) - L(x,0,t)| \leq r_1(x,t)|t|^{d_1} + s_1(x,t).
$$
\item[$(L_2)$] There is $d_2\in (0,2]$ and $r_2, s_2\in \textnormal{C}(\R\times [a,b], \R^+)$ such that 
$$
\forall (x,y,t)\in \overline{\Omega}\times \R\times[a,b],\quad \left|\frac{\partial L}{\partial x}(x,y,t) \right| \leq r_2(x,t)|y|^{d_2} + s_2(x,t).
$$
\item[$(L_3)$] There is $d_3\in (0,1]$ and $r_3,s_3\in \textnormal{C}\left(\R\times[a,b], \R^+\right)$ such that 
$$
\forall (x,y,t)\in \overline{\Omega}\times \R\times[a,b],\quad \left|\frac{\partial L}{\partial y}(x,y,t) \right| \leq r_3(x,t)|y|^{d_3} + s_3(x,t).
$$
\item[$(L_4)$] There are $\zeta >0$, $d_4 \in [1,2)$, $c_1\in \textnormal{C}(\R\times[a,b], [\zeta, \infty))$, $c_2,c_3\in \textnormal{C}([a,b], \R)$ such that 
$$
\forall (x,y,t) \in \overline{\Omega}\times \R\times[a,b],\quad L(x,y,t) \geq c_1(x,t)|y|^2 + c_2(t)|x|^{d_4} + c_3(t).
$$
\item[$(L_5)$] 
$$
\forall t\in [a,b],\quad L(\cdot, \cdot, t)\;\;\mbox{is convex}. 
$$
\end{enumerate}
Then Problem~\ref{pb} has at least one solution. In addition, if the Lagrangian $L$ is stricty convex  for any $t \in [a,b]$, 
then this solution is unique.
\end{theorem}

\subsection{G\^ateaux differentiability of $\mathcal{J}$}

We start our analysis with the following lemma: 

\begin{lemma}\label{ELlm1}
The following implications hold:
\begin{enumerate}
\item If $L$ satisfies $(L_1)$, then, for any $u\in \mathbb{H}_{0}^{\alpha, \sigma}(a,b)$, $L(u, {^{C}}\mathbb{D}_{a^+}^{\alpha, \sigma}u, t)\in \textnormal{L}^1$ and the functional $\mathcal{J}$ exists in $\R$.
\item If $L$ satisfies $(L_2)$, then, for any $u\in \mathbb{H}_{0}^{\alpha, \sigma}(a,b)$, $\frac{\partial L}{\partial x}(u, {^{C}}\mathbb{D}_{a^+}^{\alpha, \sigma}u, t) \in \textnormal{L}^1$. 
\item If $L$ satisfies $(L_3)$, then, for any $u\in \mathbb{H}_{0}^{\alpha, \sigma}(a,b)$, $\frac{\partial L}{\partial y}(u, {^{C}}\mathbb{D}_{a^+}^{\alpha, \sigma}u, t) \in \textnormal{L}^2$
\end{enumerate}
where $(L_1)-(L_3)$ are the conditions on Theorem~\ref{main1}.
\end{lemma}
\begin{proof}
Suppose that $L$ satisfies $(L_1)$ and let $u\in \mathbb{H}_{0}^{\alpha, \sigma}(a,b)  \hookrightarrow \textnormal{C}\overline{(a,b)}$. Then 
$$
|{^{C}}\mathbb{D}_{a^+}^{\alpha, \sigma} u|^{d_1} \in \textnormal{L}^{\frac{2}{d_1}}(a,b) \subset \textnormal{L}^1(a,b)
$$
and the functions 
$$
t \to r_1(u(t), t), s_1(u(t), t), |L(u(t), 0, t)| \in \textnormal{C}([a,b], \R^+) \subset \textnormal{L}^\infty(a,b) \subset \textnormal{L}^1(a,b).
$$
Moreover, $(L_1)$ implies for almost all $t\in [a,b]$
\begin{equation}\label{EL3}
|L(u(t), {^{C}}\mathbb{D}_{a^+}^{\alpha, \sigma} u(t), t)| \leq r_1(u(t), t)|{^{C}}\mathbb{D}_{a^+}^{\alpha, \sigma}u(t)|^{d_1} + s_1(u(t),t) + |L(u(t), 0, t)|.
\end{equation}
Hence, $L(u, {^{C}}\mathbb{D}_{a^+}^{\alpha, \sigma}u, t) \in \textnormal{L}^1$ and then $\mathcal{J}(u)$ exists in $\R.$

In the same way we can show (2). Now, assuming that $L$ satisfies $(L_3)$, for any $u\in \mathbb{H}_{0}^{\alpha, \sigma}(a,b)$ we have 
$$
|{^{C}}\mathbb{D}_{a^+}^{\alpha, \sigma} u|^{d_3} \in \textnormal{L}^{\frac{2}{d_3}}(a,b) \subset \textnormal{L}^2(a,b),
$$ 
Hence, as before we can show that (3) holds. 
\end{proof}

In our next result we are going to prove some differentiability properties of $\mathcal{J}$, more precisely we have:
\begin{proposition}\label{ELprop1}
Suppose that $(L_1)-(L_3)$ in Theorem~\ref{main1} hold . Then $\mathcal{J}$ is first-order G\^ateaux differentiable in any $u\in \mathbb{H}_{0}^{\alpha, \sigma}(a,b)$ and for any $u,v\in \mathbb{H}_{0}^{\alpha, \sigma}(a,b)$ we have
\begin{equation}\label{EL4}
D\mathcal{J}(u)v = \int_{a}^{b} \left[\frac{\partial L}{\partial x} \left(u, {^{C}}\mathbb{D}_{a^+}^{\alpha, \sigma}u, t\right) v + \frac{\partial L}{\partial y} \left(u, {^{C}}\mathbb{D}_{a^+}^{\alpha, \sigma}u, t\right) {^{C}}\mathbb{D}_{a^+}^{\alpha, \sigma} v\right] \;\textnormal{d}t
\end{equation} 
\end{proposition}
\begin{proof}
Let $u,v\in \mathbb{H}_{0}^{\alpha, \sigma}(a,b) \hookrightarrow \textnormal{C}\overline{(a,b)}$. For any $|h|\leq 1$ and for almost all $t\in [a,b]$ let us define the function 
$$
\psi(t,h) := L\left(u(t) + hv(t), {^{C}}\mathbb{D}_{a^+}^{\alpha, \sigma} u(t) + h {^{C}}\mathbb{D}_{a^+}^{\alpha, \sigma} v(t), t\right). 
$$
Next, we define the following function:
$$
\begin{aligned}
\Psi: [-1,1] &\to \R\\
h&\to \Psi(h) = \int_{a}^{b} L\left(u+hv, {^{C}}\mathbb{D}_{a^+}^{\alpha, \sigma}u + h{^{C}}\mathbb{D}_{a^+}^{\alpha, \sigma} v, t\right)\;\rm dt = \int_{a}^{b} \psi(t, h)\;\textnormal{d}t
\end{aligned}
$$ 
Our aim is to prove that the following limit 
$$
D\mathcal{J}(u) v = \lim_{h\to o} \frac{\mathcal{J}(u+hv) - \mathcal{J}(u)}{h} = \lim_{h\to 0}\frac{\Psi(h) - \Psi(0)}{h} = \Psi'(0)
$$
exists in $\R$. In fact, since $L\in \textnormal{C}^1$, then for almost all $t\in [a,b]$, $\psi$ is differentiable on $[-1,1]$ and for any $h\in [-1,1]$ we have 
\begin{equation}\label{EL5}
\begin{aligned}
\frac{\partial \psi}{\partial h}(t,h) &= \frac{\partial L}{\partial x}\Big(u(t) + hv(t), {^{C}}\mathbb{D}_{a^+}^{\alpha, \sigma}u(t) + h{^{C}}\mathbb{D}_{a^+}^{\alpha, \sigma}v(t), t\Big) v(t) \\
&+\frac{\partial L}{\partial y}\Big(u(t) + hv(t), {^{C}}\mathbb{D}_{a^+}^{\alpha, \sigma}u(t) + h{^{C}}\mathbb{D}_{a^+}^{\alpha, \sigma}v(t), t\Big) {^{C}}\mathbb{D}_{a^+}^{\alpha, \sigma} v(t). 
\end{aligned}
\end{equation}
Hence, by $(L_2)$ and $(L_3)$, we have for any $h\in [-1,1]$ and for almost all $t\in [a,b]$
\begin{equation}\label{EL6}
\begin{aligned}
\left|\frac{\partial \psi}{\partial h}(t,h) \right| &\leq \Big[ r_2(u(t) + hv(t), t) |{^{C}}\mathbb{D}_{a^+}^{\alpha, \sigma}u(t) + h {^{C}}\mathbb{D}_{a^+}^{\alpha, \sigma} v(t)|^{d_2} + s_2(u(t) + hv(t), t) \Big]|v(t)|\\
&+ \Big[ r_3(u(t) + hv(t), t) |{^{C}}\mathbb{D}_{a^+}^{\alpha, \sigma}u(t) + h{^{C}}\mathbb{D}_{a^+}^{\alpha, \sigma}v(t)|^{d_3} + s_3(u(t) + hv(t), t) \Big]|{^{C}}\mathbb{D}_{a^+}^{\alpha, \sigma} v(t)|.
\end{aligned}
\end{equation}
Let
$$
\begin{aligned}
&R_2 : = \max_{(t,h)\in [a,b]\times [-1,1]} r_2(u(t) +  hv(t), t) \quad \mbox{and}\quad S_2:= \max_{(t,h)\in [a,b]\times[-1,1]} s_2(u(t) + hv(t),t)\\
&R_3 : = \max_{(t,h)\in [a,b]\times [-1,1]} r_3(u(t) +  hv(t), t) \quad \mbox{and}\quad S_3:= \max_{(t,h)\in [a,b]\times[-1,1]} s_3(u(t) + hv(t),t).
\end{aligned}
$$
Then, by the analysis found in Lemma \ref{ELlm1} we have 
$$
\begin{aligned}
\left| \frac{\partial \psi}{\partial h}(t,h) \right| &\leq 2^{d_2}R_2\Big( |{^{C}}\mathbb{D}_{a^+}^{\alpha, \sigma}u(t)|^{d_2} + |{^{C}}\mathbb{D}_{a^+}^{\alpha, \sigma} v(t)|^{d_2} \Big)|v(t)| + S_2 |v(t)| \\
&+ 2^{d_3}R_3 \Big( |{^{C}}\mathbb{D}_{a^+}^{\alpha, \sigma}u(t)|^{d_{3}} + |{^{C}}\mathbb{D}_{a^+}^{\alpha, \sigma} v(t)|^{d_3} \Big)|{^{C}}\mathbb{D}_{a^+}^{\alpha, \sigma} v(t)| + S_3|{^{C}}\mathbb{D}_{a^+}^{\alpha, \sigma} v(t)| \in \textnormal{L}^1.
\end{aligned}
$$ 
Consequently, by \cite[Theorem 3 (Leibniz)]{charles}, $\Psi$ is differentiable with 
$$
\Psi'(h) = \int_{a}^{b} \frac{\partial \psi}{\partial h}(t,h)\;\rm dt\;\;\forall h\in [-1,1].
$$
Hence
\begin{multline}\label{EL7}
D\mathcal{J}(u)v = \Psi'(0) = \int_{a}^{b} \frac{\partial \psi}{\partial h}(t,0)\;\textnormal{d}t \\= \int_{a}^{b} \left[\frac{\partial L}{\partial x} (u, {^{C}}\mathbb{D}_{a^+}^{\alpha, \sigma} u, t) v + \frac{\partial L}{\partial y} (u, {^{C}}\mathbb{D}_{a^+}^{\alpha, \sigma} u, t) {^{C}}\mathbb{D}_{a^+}^{\alpha, \sigma} v \right]\;\textnormal{d}t.
\end{multline}
Now,  Lemma \ref{ELlm1} yields that 
$$
\frac{\partial L}{\partial x} (u, {^{C}}\mathbb{D}_{a^+}^{\alpha, \sigma}u, t)\in \textnormal{L}^1 \quad \mbox{and}\quad \frac{\partial L}{\partial y} (u, {^{C}}\mathbb{D}_{a^+}^{\alpha, \sigma}u, t)\in \textnormal{L}^2.
$$
Moreover, since $v\in \textnormal{C}\overline{(a,b)}$ and ${^{C}}\mathbb{D}_{a^+}^{\alpha, \sigma} v \in \textnormal{L}^2(a,b)$, then $D\mathcal{J}(u)v$ exists in $\R$ and by H\"older inequality and (\ref{conti}) we get  
$$
\begin{aligned}
| D\mathcal{J}(u)v| &\leq \|v\|_\infty \left\| \frac{\partial L}{\partial x}(u, {^{C}}\mathbb{D}_{a^+}^{\alpha, \sigma} u, \cdot) \right\|_{\textnormal{L}^1}  + \left\| \frac{\partial L}{\partial y}(u, {^{C}}\mathbb{D}_{a^+}^{\alpha, \sigma}u, \cdot) \right\|_{\textnormal{L}^2}\|{^{C}}\mathbb{D}_{a^+}^{\alpha, \sigma} v\|_{\textnormal{L}^2}\\
&\leq \left( \frac{\sqrt{\gamma (2\alpha-1, 2\sigma (b-a))}}{(2\sigma)^{\alpha-\frac{1}{2}}\Gamma (\alpha)}\left\| \frac{\partial L}{\partial x}(u, {^{C}}\mathbb{D}_{a^+}^{\alpha, \sigma}u, \cdot) \right\|_{\textnormal{L}^1} + \left\| \frac{\partial L}{\partial y}(u, {^{C}}\mathbb{D}_{a^+}^{\alpha, \sigma}u, \cdot) \right\|_{\textnormal{L}^2} \right) \|v\|.
\end{aligned}
$$
Consequently, $D\mathcal{J}(u)$ is linear and continuous from $\mathbb{H}_{0}^{\alpha, \sigma}(a,b)$ to $\R$.
\end{proof}

In the following theorem we prove that any critical point of $\mathcal{J}$ is a weak solution of (\ref{EL}).
\begin{theorem}\label{ELtm1}
 If $u$ is a weak extremal of $\mathcal{J}$, then $u$ is a weak solution of (\ref{EL}) for all $u\in \mathbb{H}_{0}^{\alpha, \sigma}(a,b)$ and for all $v\in\textnormal{C}_{0}^{\infty}(a,b)$. 
\end{theorem}

In order to prove the above theorem, we first need
to show the  du Bois–Reymond lemma for Problem~\ref{pb}.

\begin{theorem}(Generalized du Bois--Reymond lemma)\label{DR}
Set $f,g \in \textnormal{L}^{ 1}( [a,b])$. If
\begin{equation}\label{ldM}
	\int_{a}^{b}\left[f(t)\eta(t) +g(t) {^{C}}\mathbb{D}_{a^+}^{\alpha, \sigma} \eta(t)\right]\,\textnormal{d}t=0	
\end{equation}
for any $\eta \in \textnormal{L}^1(a,b))$ with $\eta(a)=\eta(b)=0$. Then
\begin{equation}\label{reprM}
	\mathbb{D}_{b^-}^{\alpha, \sigma}g(t)=-f(t),\,\, t\in [a,b]\,\, \textnormal{almost everywhere}.
\end{equation}
\end{theorem}

\begin{proof}
Let's proceed by density. For any $\eta \in\mathcal{C}^\infty_0((a,b))\subset \textnormal{L}^1(a,b)$, we have from \eqref{ldM} that,
\begin{equation}\label{RB1} 
\begin{array}{lllll}
	0&=&\dis\int_{a}^{b}\left[f(t)\eta(t) +g(t)  {^{C}}\mathbb{D}_{a^+}^{\alpha, \sigma}\eta(t)\right]\,\textnormal{d}t\\
	&=&\dis	\int_{a}^{b}\left[e^{-\sigma t}f(t)e^{\sigma t}\eta(t) +g(t)  {^{C}}\mathbb{D}_{a^+}^{\alpha, \sigma}\eta(t)\right]\,\textnormal{d}t\,.
\end{array}\end{equation}
Since $e^{-\sigma t}f\in \textnormal{L}^1(a,b)$, let $z\in \textnormal{AC}([a,b])$ be the representative of $e^{-\sigma t}f$ i.e.,
$$z(t)=\textnormal{cst}+\int_a^t e^{-\sigma \tau}f(t)\,\textnormal{d}t.$$
So, for all $ \eta\in \textnormal{C}^1_0[a,b]$, using the classical integration by parts formula, Theorem~\ref{IPPD} and \eqref{T04}, one obtains
\begin{equation*}
 \begin{array}{lllll}
\eqref{RB1}\Leftrightarrow	0&=&\dis	\int_{a}^{b}\left[-z(t) \dfrac{\textnormal{d}}{\textnormal{d}t} (e^{\sigma t}\eta(t)) +\mathbb I^{1-\alpha,\sigma}_{b^-}\left(e^{-\sigma t}g(t)\right)\dfrac{\textnormal{d}}{\textnormal{d}t}(e^{\sigma t} \eta(t))\right]\,\textnormal{d}t\\
		&=&\dis	\int_{a}^{b}\left[-z(t)  +\mathbb I^{1-\alpha,\sigma}_{b^-}\left(e^{-\sigma t}g(t)\right)\right]\dfrac{\textnormal{d}}{\textnormal{d}t}(e^{\sigma t} \eta(t))\,\textnormal{d}t\,.
\end{array}
\end{equation*}
By the classical du Bois--Reymond lemma \cite[Lemma 1.8]{GB}, we have
$$-z(t)  +\mathbb I^{1-\alpha,\sigma}_{b^-}\left(e^{-\sigma t}g(t)\right)=c\in \R\,,$$
 which implies that 
 $$-e^{\sigma t}\dfrac{\textnormal{d}}{\textnormal{d}t}\left[-z(t)  +\mathbb I^{1-\alpha,\sigma}_{b^-}\left(e^{-\sigma t}g(t)\right)\right]=0 \Leftrightarrow \eqref{reprM}$$ as asserted.
\end{proof}

\begin{proof}(of Theorem~\ref{ELlm1})
Let $u$ be a weak extremal of $\mathcal{J}$. Then, by density for any $v\in \textnormal{C}_{0}^{\infty}(a,b)$, we obtain
$$
D\mathcal{J}(u)v = \int_{a}^{b} \left[ \frac{\partial L}{\partial x}\left(u, {^{C}}\mathbb{D}_{a^+}^{\alpha, \sigma}u, t\right) v + \frac{\partial L}{\partial y}\left(u, {^{C}}\mathbb{D}_{a^+}^{\alpha, \sigma}u, t\right) {^{C}}\mathbb{D}_{a^+}^{\alpha, \sigma}v \right]\,\textnormal{d}t = 0,
$$
and using  Theorem \ref{DR}, we get 
$$
\frac{\partial L}{\partial x}\left(u, {^{C}}\mathbb{D}_{a^+}^{\alpha, \sigma}u, t\right) + \mathbb{D}_{b^-}^{\alpha, \sigma}\left(\frac{\partial L}{\partial y}\left(u, {^{C}}\mathbb{D}_{a^+}^{\alpha, \sigma}u, t\right) \right) = 0\quad \mbox{a.e. all  $t\in[a,b]$},
$$
and then $u\in \mathbb{H}_{0}^{\alpha, \sigma}(a,b) \subset \textnormal{C}\overline{(a,b)}$ satisfies (\ref{EL}) a.e. on $[a,b]$. The proof is complete. 
\end{proof}

\subsection{Existence of a  minimizer of Problem~\ref{pb}}

Next Lemma is useful to prove our Theorem~\ref{main1}.

\begin{lemma}\label{GMlm1}
Suppose that $(L_4)$ in Theorem~\ref{main1} holds. Then $\mathcal{J}$ is coercive, that is 
\begin{equation}\label{coer}
\lim\limits_{\substack{\Vert u \Vert\to \infty \\ u \in \overline{\Omega} }} {\cal J}(u) = +\infty\,.
\end{equation}
\end{lemma}
\begin{proof}
Let $u\in \mathbb{H}_{0}^{\alpha, \sigma}(a,b)\subset \overline{\Omega}$, then $(L_4)$ yields that 
$$
\mathcal{J}(u) = \int_{a}^{b} L(u, {^{C}}\mathbb{D}_{a^+}^{\alpha, \sigma}u, t)\,\textnormal{d}t \geq \int_{a}^{b}\Big[ c_1(u,t)|{^{C}}\mathbb{D}_{a^+}^{\alpha, \sigma}u|^2 + c_2(t)|u|^{d_4} + c_3(t)\Big]\;\textnormal{d}t
$$
Combining H\"older inequality, Lemma \ref{Blm01} and Theorem \ref{Lp} we have 
$$
\begin{aligned}
\|u\|_{\textnormal{L}^{d_4}}^{d_4} &\leq (b-a)^{1-\frac{d_4}{2}}\|\mathbb{I}_{a^+}^{\alpha, \sigma}{^{C}}\mathbb{D}_{a^+}^{\alpha, \sigma}u\|_{\textnormal{L}^2}^{d_4}\\
&\leq (b-a)^{1-\frac{d_4}{2}} \left(\frac{\gamma (\alpha, \sigma(b-a))}{\sigma^\alpha \Gamma (\alpha)} \right)^{d_4}\|u\|^{d_4}.
\end{aligned}
$$
Consequently 
$$
\begin{aligned}
\mathcal{J}(u) &\geq \zeta \|{^{C}}\mathbb{D}_{a^+}^{\alpha, \sigma}u\|_{\textnormal{L}^2}^{2} - \|c_2\|_\infty\|u\|_{\textnormal{L}^{d_4}}^{d_4} - (b-a)\|c_3\|_\infty\\
&\geq \zeta \|u\|^2 - \|c_2\|_\infty (b-a)^{1-\frac{d_4}{2}}\left(\frac{\gamma(\alpha, \sigma (b-a))}{\sigma^\alpha \Gamma (\alpha)} \right)^{d_4}\|u\|^{d_4} - (b-a)\|c_3\|_\infty.
\end{aligned}
$$
Since $d_4 < 2$, the proof is complete.
\end{proof}

\begin{remark}\label{inf}	
\vline
	\begin{enumerate}	
		\item[(i)] The coercivity condition $\eqref{coer}$ ensures that
		 infimum $\textnormal{Inf}\{{\cal J}(u): u\in  \overline{\Omega}\}$ is finite.\\
		\item[(ii)] Obviously, the degenerate case where ${\cal J}(u)\equiv +\infty, \, \forall u\in  \overline{\Omega}$, is to rule out in the {\rm Theorem \ref{main1}}.
	\end{enumerate}
\end{remark}

Next, we state a more practical result to ensure the coerciveness of $\cal J$.

\begin{lemma}\label{coerc}
	Let  $L: \overline{\Omega} \times\R\times [a,b]\longrightarrow \R$ be a continuous function. If $L$ satisfies the conditions
	\begin{equation}\label{coerci}
		\alpha\vert y\vert\leq L(x,y,t)\leq\beta(\vert x\vert+\vert y \vert)\,, 
	\end{equation}with $\alpha,\beta>0\,,$ 
	then $\cal J$ is coercive over  $\mathbb{H}_{0}^{\alpha, \sigma}(a,b) \,.$
\end{lemma}

To prove the previous lemma we need the following useful theorem:

\begin{theorem}\label{PCAI}[Fractional Poincaré inequality]\label{poin} Let $X\in \R$ an open, bounded and connected set. Then, there exist a constant $C_X>0$ such that
	\begin{equation}\label{poin}
		\Vert u\Vert_{\textnormal{L}^2(X)}\leq C_X\left\Vert  {^{C}}\mathbb{D}_{a^+}^{\alpha, \sigma}u \right\Vert_{\L^2(X)}
	\end{equation}
	for all functions $u\in \W_0^{1,2} (X)\,.$  	
\end{theorem}

\begin{proof}
Let's proceed by contradiction. Suppose that inequality \eqref{poin} is not verified. We can then consider a sequence $(u_n)_{n \in \mathbb{N}}\in\W_0^{1,2} (X) $ satisfying the following conditions for any $n$:
	\begin{eqnarray}
		\Vert u_n\Vert_{\L^2(X)}&=&\mathbf{1}; \label{poin1}\\
		\Vert u_n\Vert_{\L^2(X)}&\geq& n\left\Vert  {^{C}}\mathbb{D}_{a^+}^{\alpha, \sigma}u_n \right\Vert_{\L^2(X)}\,.\label{poin2}
	\end{eqnarray}
	From \eqref{poin1} and \eqref{poin2}, one obtains
	\begin{equation}
		\left\Vert  {^{C}}\mathbb{D}_{a^+}^{\alpha, \sigma}u_n \right\Vert_{\L^2(X)}\leq \dfrac{1}{n}\,,
	\end{equation}
	which proves, on the one hand that $ {^{C}}\mathbb{D}_{a^+}^{\alpha, \sigma}u_n\longrightarrow 0$ on $\L^2(X)$, and on the other hand \newline that  the sequence $(u_n)_{n \in \mathbb{N}}$ is bounded on $\W^{1,2} (X)$. By Rellich–Kondrachov's  theorem \cite{MR2759829}, we can assume (modulo an extraction) that the sequence $(u_n)_{n \in \mathbb{N}}$ converges on $\L^2(X)$ to a function $u\in \L^2(X)$. Looking at all this in the sense of distributions, we have that $u_n\longrightarrow u$ on $\mathcal{D'}(X)$ and $ {^{C}}\mathbb{D}_{a^+}^{\alpha, \sigma}u_n\longrightarrow 0$ on $\mathcal{D'}(X)$. As we also have $ {^{C}}\mathbb{D}_{a^+}^{\alpha, \sigma}u_n\longrightarrow {^{C}}\mathbb{D}_{a^+}^{\alpha, \sigma}u$ on $\mathcal{D'}(X)$, we obtain ${^{C}}\mathbb{D}_{a^+}^{\alpha, \sigma}u=0$. Since $X$ is connected and $u\in \W_0^{1,2} (X)$, and taking into account \eqref{T05}, we deduce \bigskip that $u=0$. This contradicts \eqref{poin1}. 
\end{proof}

\begin{proof}(of Lemma~\ref{coerc})
	It is sufficient to combine \eqref{coerci} and Theorem~\ref{PCAI}.	
\end{proof}

Now we are in position to give the proof of Theorem \ref{main1}

\begin{proof}(of Theorem \ref{main1})

Since the set $\{\mathcal{J}(u)\,|\,u\in \overline{\Omega} \}$ is a bounded non empty subset of $\R$, we can set down any minimizing sequence $(u_n)_{n\in \N}$ in $\mathbb{H}_{0}^{\alpha, \sigma}(a,b)$ for $\mathcal{J}$, i.e., 
$$
\mathcal{J}(u_n) \to \inf_{\bar{u}\in \mathbb{H}_{0}^{\alpha, \sigma}(a,b)}\mathcal{J}(\bar{u}):=m.
$$
Lemma \ref{ELlm1} yields that $\mathcal{J}(u)\in \R$ for any $u\in \mathbb{H}_{0}^{\alpha, \sigma}(a,b)$, hence $m< \infty$. Moreover, Lemma \ref{GMlm1} implies that $(u_n)_{n\in \N}$ is bounded in $\mathbb{H}_{0}^{\alpha, \sigma}(a,b)$. Since $\mathbb{H}_{0}^{\alpha, \sigma}(\R)$ is a Hilbert space then its reflexive, so there is $u\in \mathbb{H}_{0}^{\alpha, \sigma}(a,b)$ such that up to a subsequence we have 
$$
u_n \rightharpoonup u_0\quad \mbox{in $\mathbb{H}_{0}^{\alpha, \sigma}(a,b)$}.
$$
By Theorem \ref{compact}, $u_n \to u_0$ in $\textnormal{C}\overline{(a,b)}$. 

Note that, by convexity, for any $n\in \mathbb{N}$ we derive 

\begin{multline}
	\mathcal{J}(u_n) \geq \mathcal{J}(u_0) + \mathcal{J}'(u_0)(u_n-u_0) 
	= \int_{a}^{b}L\left(u_0, {^{C}}\mathbb{D}_{a^+}^{\alpha, \sigma}u_0, t\right)\,\rm dt\\+ \int_{a}^{b}\left[\frac{\partial L}{\partial x}\left(u_0, {^{C}}\mathbb{D}_{a^+}^{\alpha, \sigma}u_0, t\right)(u_n-u_0) + \frac{\partial L}{\partial y}\left(u_0, {^{C}}\mathbb{D}_{a^+}^{\alpha, \sigma}u_0, t\right)\Big({^{C}}\mathbb{D}_{a^+}^{\alpha, \sigma} u_n - {^{C}}\mathbb{D}_{a^+}^{\alpha, \sigma}u_0\Big) \right]\textnormal{d}t.
\end{multline}

Therefore, using Lemma \ref{GMlm1}, the weak convergence in $\mathbb{H}_{0}^{\alpha, \sigma}(a,b)$, the strong convergence in $\textnormal{C}\overline{(a,b)}$ and making $n\to \infty$, we obtain 
$$
m : = \inf_{\bar{u}\in \mathbb{H}_{0}^{\alpha, \sigma}(a,b)}\mathcal{J}(\bar{u}) \geq \int_{a}^{b}L\left(u_0,{^{C}}\mathbb{D}_{a^+}^{\alpha, \sigma} u_0, t\right)\,\textnormal{d}t = \mathcal{J}(u_0).
$$
Consequently, $u_0$ is a  minimizer of Problem~\ref{pb}.

Now, let us prove that $u_0$ is unique if the Lagrangian $L$ is stricty convex  for any $t \in [a,b]$.

Suppose there exist $u_1,u_2 \in \mathbb{H}_{0}^{\alpha, \sigma}(a,b)$  with $u_1\neq u_2$ such that  $$\textnormal{Inf}\{{\mathcal{J}}(u): u\in \mathbb{H}_{0}^{\alpha, \sigma}(a,b) \}=\mathcal{J}(u_1)=\mathcal{J}(u_2)=m<\infty.$$
Let $\tilde{u}=\frac{u_1+u_2}{2}.$ 
So $\tilde{u}\in \mathbb{H}_{0}^{\alpha, \sigma}(a,b)$ and since $L$ is convex, we
can say that $\tilde{u}$ is also a minimum of $\mathcal{J}$ because
$$m\leq\mathcal{J}(\tilde{u})\leq\frac{\mathcal{J}(u_1)}{2}+\frac{\mathcal{J}(u_2)}{2}=m.$$
We thus obtain

\begin{multline*}
	\int_{a}^{b}\left[\dfrac{L\left(u_1(t),{^{C}}\mathbb{D}_{a^+}^{\alpha, \sigma} u_1(t),t\right)}{2}+\dfrac{L\left(u_2(t),{^{C}}\mathbb{D}_{a^+}^{\alpha, \sigma} u_2(t),t\right)}{2}-\right.\\\left.-L\left(\dfrac{u_1(t)+u_2(t)}{2},\dfrac{{^{C}}\mathbb{D}_{a^+}^{\alpha, \sigma} u_1(t)+{^{C}}\mathbb{D}_{a^+}^{\alpha, \sigma}u_2(t)}{2},t\right)\right]\,\textnormal{d} t=0\,.
\end{multline*}
$(L_5)$ ensures that the integrant is non-negative. As the integral
is null, so the only possibility is
\begin{multline*}
	\dfrac{L\left(u_1(t),{^{C}}\mathbb{D}_{a^+}^{\alpha, \sigma} u_1(t),t\right)}{2}+\dfrac{L\left(u_2(t),{^{C}}\mathbb{D}_{a^+}^{\alpha, \sigma} u_2(t),t\right)}{2}-\\-L\left(\dfrac{u_1(t)+u_2(t)}{2},\dfrac{{^{C}}\mathbb{D}_{a^+}^{\alpha, \sigma} u_1(t)+{^{C}}\mathbb{D}_{a^+}^{\alpha, \sigma}u_2(t)}{2},t\right)=0,\,\forall t\in [a,b]\,.
\end{multline*}
We now use the fact that Lagrangean $L$ is stricty convex  for any    to get that $u_1=u_2$ and ${^{C}}\mathbb{D}_{a^+}^{\alpha, \sigma}u_1={^{C}}\mathbb{D}_{a^+}^{\alpha, \sigma} u_2 $, $t\in [a,b]$, as asserded.	

The proof is completed.		
\end{proof} 

\subsection{Fractional tempered Noether theorem}

In the seminal paper written in 1918 \cite{Noether:1918}, E. Noether proved that there is a one-to-one correspondence between symmetry groups of a variational problem and conservation laws of its Euler–Lagrange equations. The following definition is crucial to our goal to prove Noether’s theorem(see Theorem~\ref{CNT}).

\begin{definition}(\textbf{Variational symmetry group})\label{sg}
	Let $\textnormal{D}\subseteq \mathbb{R}^n$ be an open set. We say that $ \Phi=\{\psi(s,\cdot)\}_{s\in \mathbb{R}}:\textnormal{D}\longrightarrow\textnormal{D}$ 
	is a one parameter group of diffeomorphisms of $\textnormal{D}$ if it satisfies:
	\begin{itemize}
		\item[(i)] $\psi(s,\cdot)$ is of class $\textnormal{C}^2$ 
		 with respect to $s$;
		\item[(ii)] For each $s\in \mathbb{R}$, the map $\psi(s,\cdot)$ is invertible, and $\psi(s,\cdot)^{-1}\in \textnormal{C}^2(\textnormal{D},\textnormal{D})$;
		\item[(iii)] $\psi(0,\cdot)=\textnormal{Id}_{\textnormal{D}}$,  where \textnormal{Id} is the identity function;
		\item[(iv)] $\forall s,s'\in \mathbb{R}:s+s'\in \mathbb{R}\Rightarrow \psi(s,\cdot)\circ \psi(s',\cdot)=\psi(s+s',\cdot)$.
	\end{itemize}
\end{definition}
The translation in a spatial direction $h$ is a typical case of a one-parameter group of diffeomorphisms:
$$\psi: q\mapsto q+sh,\quad q,h\in \mathbb{R}^n\,.$$
Another classical example is the rotations by an angle $\theta$
$$\psi: q\mapsto qe^{is\theta}, \,\,q\in \mathbb{C},\,\,\theta\in\mathbb{R}\,.$$

In the present work, we use the concept of a group of diffeomorphisms rather than the concept of infinitesimal transformations as in \cite{FT}. These two concepts (see the classical book \cite[Chapter 4]{PO}  for more details) can be related by a Taylor expansion of $\psi(s,q(t))$ in the neighborhood of $s = 0$. We have:
$$\psi(s,q(t))=\psi(0,q(t))+s\frac{\partial \psi}{\partial s}(0,q(t))\,.$$
Therefore, for a sufficiently small infinitesimal $s$, we can always say that $\psi(s, t)$ is infinitely close to a transformation of the form $q(t)\mapsto q(t)+s\eta(q(t))$, where $\eta(t):=\frac{\partial \psi}{\partial s}(0,q(t))$.

Before we can discuss Noether’s theorem in detail, we need to discuss what we mean by a symmetry of a system. What we use to describe the system are the equations of motion, so it is natural to say that a symmetry transformation of a system is a transformation of the dependent and independent variables that leaves the explicit form of the equations of motion unchanged. To arrive at this, we introduce the following definition: 

\begin{definition}(Invariance condition (IC))\label{def:inv}
	Let  $S=\{\xi(s,\cdot)\}_{s\in \mathbb{R}}\in \textnormal{C}^2(S'\times S',S')$ be one parameter groups of diffeomorphims on the open set $S'\subseteq \mathbb{R}$. Fractional functional $\mathcal{J}$ is said to be \emph{invariant under the action} $S$, if for any weak extremal $u\in \textnormal{C}^2([a,b], S')$ it satisfies
	\begin{equation}\label{invg}
		\int_{t_a}^{t_b}L\left(u(t),{^{C}}\mathbb{D}_{a^+}^{\alpha, \sigma}u(t),t\right){\rm d}t=\int_{t_a}^{t_b}L\left(\xi(s,u(t),t),{^{C}}\mathbb{D}_{a^+}^{\alpha, \sigma}\xi\left(s,u(t),t\right),t\right){\rm d}t
	\end{equation}
	for all $s\in \mathbb{R}$ and all $t\in [t_a,t_b]\subseteq [a,b]$.
\end{definition}

Notice that, from Definition~\ref{sg}, we have that
$$L\left(\xi(s,u(t),t),{^{C}}\mathbb{D}_{a^+}^{\alpha, \sigma}\xi\left(s,u(t),t\right),t\right)\in \L^1$$
consequently \eqref{invg} is well-defined.

The next Lemma is very useful to check for $s$-invariance. It is also useful to determine the $s$-invariance transformations.

\begin{lemma}(Necessary condition of invariance )
Let $\mathcal{J}$ be a fractional Lagrangian functional given by \eqref{EL1}  invariant in the sense of Definition~\ref{def:inv}, then
\begin{multline}
	\label{eq:cnsidf11} 
	\frac{\partial L}{\partial y}\left(u(t),{^{C}}\mathbb{D}_{a^+}^{\alpha, \sigma}u(t),t\right)
	 {{^{C}}\mathbb{D}_{a^+}^{\alpha, \sigma} \frac{\partial \xi}{\partial s}(0,u(t),t)}\\
	- \frac{\partial \xi}{\partial s}(0,u(t),t)
	\mathbb{D}_{b^-}^{\alpha; \sigma}\frac{\partial L}{\partial y}\left(u(t),{^{C}}\mathbb{D}_{a^+}^{\alpha, \sigma}u(t),t\right) = 0 \, .
\end{multline}
\end{lemma}

\begin{proof}
As condition \eqref{invg} is valid for any subinterval
$[{t_{a}},{t_{b}}] \subseteq [a,b]$, we can get rid of the integral
signs in \eqref{invg}, i.e.
$$L\left(u(t),{^{C}}\mathbb{D}_{a^+}^{\alpha, \sigma}u(t),t\right)=L\left(\xi(s,u(t),t),{^{C}}\mathbb{D}_{a^+}^{\alpha, \sigma}\xi\left(s,u(t),t\right),t\right).$$ Differentiating this condition with respect to $s$ at $s=0$,
the usual chain rule for the classical derivatives implies

\begin{multline}
	\label{eq:SP} 0 = \frac{\partial L}{\partial x}\left(u(t),{^{C}}\mathbb{D}_{a^+}^{\alpha, \sigma}u(t),t\right)
	\frac{\partial \xi}{\partial s}(0,u(t),t)\\
	+ \frac{\partial L}{\partial y}\left(u(t),{^{C}}\mathbb{D}_{a^+}^{\alpha, \sigma}u(t),t\right) \frac{\partial}{\partial s}\left[ {^{C}}\mathbb{D}_{a^+}^{\alpha, \sigma}
	\xi(s,u,t)\right]\vert_{s=0} \,.
\end{multline}
Recalling that $\xi(s,u,t)\in \textnormal{C}^2$ with respect to $s$ in the Definition~\ref{def:inv}, and since ${^{C}}\mathbb{D}_{a^+}^{\alpha, \sigma}$ acts on $t$ and $\frac{\partial }{\partial s}$ on variable $s$, we deduce that

 \begin{equation}
	\label{ce1}
	\frac{\partial}{\partial s}\left[{^{C}}\mathbb{D}_{a^+}^{\alpha, \sigma} \xi(s,u,t)\right]\mid_{s=0}
	={^{C}}\mathbb{D}_{a^+}^{\alpha, \sigma}\frac{\partial\xi}{\partial s}(0,u,t)\,.
\end{equation}
Substituting \eqref{ce1} into \eqref{eq:SP} and using the  fractional Euler-Lagrange equation \eqref{EL}, we obtain the desired conclusion \eqref{eq:cnsidf11}.
\end{proof}

\begin{definition}(Constant of motion)\label{lc1}
	Let $X'\subseteq[a,b]$ be an open set. We say that a quantity $C\left(u(t),{^{C}}\mathbb{D}_{a^+}^{\alpha, \sigma}u(t),t\right)$ is a \emph{constant of motion} on $X',S'$ if
	\begin{equation}\label{lcon}
		\frac{{\rm d}C}{{\rm d}t}\left(u(t),{^{C}}\mathbb{D}_{a^+}^{\alpha, \sigma}u(t),t\right)=0\quad \forall t\in X'
	\end{equation} 
	along all the extremals $u\in \textnormal{C}^2\left([a,b],S'\right)$.
\end{definition}

\begin{theorem}(Tempered fractional Noether's theorem) \label{CNT}
	If the functional $\mathcal{J}$ is invariant in the sense of Definition~\ref{def:inv},  then the quantity $C[u]_t$ defined for all extremal $u\in \textnormal{C}^2\left([a,b],S'\right)$ and $t\in X'$ by
	\begin{multline}\label{NCL}
		C[u]_t=
		\frac{\partial L}{\partial y}[u]_t\mathbb{I}_{a^+}^{1-\alpha, \sigma}\frac{\partial \xi}{\partial s}(0,u(t),t) + \frac{\partial \xi}{\partial s}(0,u(t),t) \mathbb{I}_{b^-}^{1-\alpha, \sigma}\frac{\partial L}{\partial y}[u]_t\\
		=\lim_{t\to b^-}\left(\frac{\partial L}{\partial y}[u]_t\mathbb{I}_{a^+}^{1-\alpha, \sigma}\frac{\partial \xi}{\partial s}(0,u(t),t)\right) -\lim_{t\to a^+}\left( \frac{\partial \xi}{\partial s}(0,u(t),t) \mathbb{I}_{b^-}^{1-\alpha, \sigma}\frac{\partial L}{\partial y}[u]_t
		\right)
	\end{multline}
	is a constant of motion on $X'$ and $S'$. Here the operator $[u]_t$ is defined as $[u]_t=\left(u(t),{^{C}}\mathbb{D}_{a^+}^{\alpha, \sigma}u(t),t\right)$.
	
\end{theorem}

\begin{proof}
	Let $\alpha \in (0,1)$ and $\sigma >0$. If $h,v\in C^1[a,b]$, then we have 
	$$
	\begin{aligned}
		\left(\frac{d}{dx} + \sigma \right) [h(x)v(x)] &= \left(\frac{d}{dx} + \sigma \right)h(x) v(x) + h(x)\frac{d}{dx}v(x)\\
		&= \frac{d}{dx}h(x) v(x) + h(x)\left(\frac{d}{dx} + \sigma \right)v(x).
	\end{aligned}
	$$
	Hence
	\begin{equation}\label{TF01}
		\left(\frac{d}{dx} + \sigma \right) h(x) v(x) = \left(\frac{d}{dx} + \sigma \right)\Big[h(x)v(x) \Big] - h(x)\frac{d}{dx}v(x)
	\end{equation}
	and
	\begin{equation}\label{TF02}
		h(x) \left(\frac{d}{dx} + \sigma \right)v(x) = \left(\frac{d}{dx}+\sigma \right)\Big[h(x) v(x) \Big] - \frac{d}{dx}h(x) v(x).
	\end{equation}
	Moreover, we note that 
	$$
	\begin{aligned}
		\mathbb{D}_{a^+}^{\alpha, \sigma}u(x) &= e^{-\sigma x} {_{a}}D_{x}^{\alpha}e^{\sigma x}u(x)\\
		&= \left(\frac{d}{dx} + \sigma \right) \mathbb{I}_{a^+}^{1-\alpha, \sigma}u(x)
	\end{aligned}
	$$
	and 
	$$
	\begin{aligned}
		\mathbb{D}_{b^-}^{\alpha, \sigma}u(x) &= e^{\sigma x} {_{x}}D_{b}^{\alpha} e^{-\sigma x}u(x)\\
		&= - \left(\frac{d}{dx} + \sigma \right) \mathbb{I}_{b^-}^{1-\alpha, \sigma}u(x)
	\end{aligned}
	$$
	By other hand, by (\ref{TF01}) we derive
\begin{equation}\label{TF03}
	\begin{aligned}
		&\frac{\partial L}{\partial y}{^{C}}\mathbb{D}_{a^+}^{\alpha, \sigma} \frac{\partial \xi}{\partial s}(0,u(t),t)  = \frac{\partial L}{\partial y}\mathbb{D}_{a^+}^{\alpha, \sigma}\frac{\partial \xi}{\partial s}(0,u(t),t) - \frac{\frac{\partial \xi}{\partial s}(0,u(a),a)}{\Gamma (1-\alpha)}(t-a)^{-\alpha}e^{-\sigma (t-a)} \frac{\partial L}{\partial y}\\
		&= \left(\frac{\rm d}{{\rm d}t} + \sigma \right)\Big(\frac{\partial L}{\partial y}\mathbb{I}_{a^+}^{1-\alpha}\frac{\partial \xi}{\partial s}(0,u(t),t)  \Big) - \mathbb{I}_{a^+}^{1-\alpha, \sigma}\frac{\partial \xi}{\partial s}(0,u(t),t)\frac{\rm d}{{\rm d}t}\frac{\partial L}{\partial y}\\ &- \frac{\frac{\partial \xi}{\partial s}(0,u(a),a)}{\Gamma (1-\alpha)}(t-a)^{-\alpha}e^{-\sigma (t-a)}\frac{\partial L}{\partial y}
	\end{aligned}
\end{equation}
and by (\ref{TF02}) we derive 
\begin{equation}\label{TF04}
	\begin{aligned}
		\frac{\partial \xi}{\partial s}(0,u(t)) \mathbb{D}_{b^-}^{\alpha, \sigma}\frac{\partial L}{\partial y} &= -\frac{\partial \xi}{\partial s}(0,u(t))\left(\frac{\rm d}{{\rm d}t} + \sigma \right) \mathbb{I}_{b^-}^{1-\alpha, \sigma}\frac{\partial L}{\partial y}\\
		&= - \left(\frac{\rm d}{{\rm d}t} + \sigma \right) \Big( \frac{\partial \xi}{\partial s}(0,u(t)) \mathbb{I}_{b^-}^{1-\alpha, \sigma}\frac{\partial L}{\partial y} \Big) + \frac{\rm d}{{\rm d}t}\frac{\partial \xi}{\partial s}(0,u(t)) \mathbb{I}_{b^-}^{1-\alpha, \sigma}\frac{\partial L}{\partial y}.
	\end{aligned}
\end{equation}
Consequently, combining (\ref{TF03}) and (\ref{TF04}), we obtain 
\begin{equation}\label{TF10}
\begin{aligned}
&\frac{\partial L}{\partial y}{^{C}}\mathbb{D}_{a^+}^{\alpha, \sigma} \frac{\partial \xi}{\partial s}(0,u(t),t) - \frac{\partial \xi}{\partial s}(0,u(t)) \mathbb{D}_{b^-}^{\alpha, \sigma}\frac{\partial L}{\partial y}\\&= \left(\frac{\rm d}{{\rm d}t} + \sigma \right) \Bigg[\frac{\partial L}{\partial y} \mathbb{I}_{a^+}^{1-\alpha, \sigma}\frac{\partial \xi}{\partial s}(0,u(t),t)  + \frac{\partial \xi}{\partial s}(0,u(t),t) \mathbb{I}_{b^-}^{1-\alpha, \sigma}\frac{\partial L}{\partial y} \Bigg]\\&
	-\mathbb{I}_{a^+}^{1-\alpha, \sigma}\frac{\partial \xi}{\partial s}(0,u(t),t) \frac{d}{dx}\frac{\partial L}{\partial y} - \frac{\rm d}{{\rm d}t}\frac{\partial \xi}{\partial s}(0,u(t),t) \mathbb{I}_{b^-}^{1-\alpha, \sigma}\frac{\partial L}{\partial y}
	- \frac{\frac{\partial \xi}{\partial s}(0,u(a),a)}{\Gamma (1-\alpha)}(t-a)^{-\alpha}e^{-\sigma (t-a)} \frac{\partial L}{\partial y}
	\\&
	= \frac{\rm d}{{\rm d}t} \Bigg[ \mathbb{I}_{a^+}^{1-\alpha, \sigma}\frac{\partial \xi}{\partial s}(0,u(t),t) \frac{\partial L}{\partial y} + \frac{\partial \xi}{\partial s}(0,u(t),t) \mathbb{I}_{b^-}^{1-\alpha, \sigma}\frac{\partial L}{\partial y} \Bigg]\\&
	+\sigma \Bigg[ \mathbb{I}_{a^+}^{1-\alpha, \sigma}\frac{\partial \xi}{\partial s}(0,u(t),t) \frac{\partial L}{\partial y} + \frac{\partial \xi}{\partial s}(0,u(t),t) \mathbb{I}_{b^-}^{1-\alpha, \sigma}\frac{\partial L}{\partial y} \Bigg]\\&
	-\mathbb{I}_{a^+}^{1-\alpha, \sigma}\frac{\partial \xi}{\partial s}(0,u(t),t) \frac{\rm d}{{\rm d}t}\frac{\partial L}{\partial y}- \frac{\rm d}{{\rm d}t}\frac{\partial \xi}{\partial s}(0,u(t),t) \mathbb{I}_{b^-}^{1-\alpha, \sigma}\frac{\partial L}{\partial y}- \frac{\frac{\partial \xi}{\partial s}(0,u(a),a)}{\Gamma (1-\alpha)}(t-a)^{-\alpha}e^{-\sigma (t-a)} \frac{\partial L}{\partial y}\,.
\end{aligned}
\end{equation}
Note that 
$$
\begin{aligned}
	&\int_{a}^{b} \left(\sigma \Bigg[ \frac{\partial L}{\partial y}\mathbb{I}_{a^+}^{1-\alpha, \sigma}\frac{\partial \xi}{\partial s}(0,u(t),t)  + \frac{\partial \xi}{\partial s}(0,u(t),t) \mathbb{I}_{b^-}^{1-\alpha, \sigma}\frac{\partial L}{\partial y} \Bigg]\right.\\&\left. -\mathbb{I}_{a^+}^{1-\alpha, \sigma} \frac{\partial \xi}{\partial s}(0,u(t),t) \frac{\rm d}{{\rm d}t}\frac{\partial L}{\partial y} - \frac{\rm d}{{\rm d}t} \frac{\partial \xi}{\partial s}(0,u(t),t) \mathbb{I}_{b^-}^{1-\alpha, \sigma}\frac{\partial L}{\partial y} \right) {\rm d}t\\
	&= \int_{a}^{b} \left(\sigma \mathbb{I}_{a^+}^{1-\alpha, \sigma} \frac{\partial \xi}{\partial s}(0,u(t),t) \frac{\partial L}{\partial y} - \mathbb{I}_{a^+}^{1-\alpha, \sigma} \frac{\partial \xi}{\partial s}(0,u(t),t) \frac{\rm d}{{\rm d}t}\frac{\partial L}{\partial y} \right) {\rm d}t \\&+ \int_{a}^{b} \left(\sigma  \frac{\partial \xi}{\partial s}(0,u(t),t)\mathbb{I}_{b^-}^{1-\alpha, \sigma}\frac{\partial L}{\partial y} - \frac{\rm d}{{\rm d}t} \frac{\partial \xi}{\partial s}(0,u(t),t) \mathbb{I}_{b^-}^{1-\alpha, \sigma}\frac{\partial L}{\partial y} \right) {\rm d}t
\end{aligned}
$$
By using integration by parts, the first integral can be written as:
\begin{equation}
\begin{aligned}
	&\int_{a}^{b}\left(\sigma \mathbb{I}_{a^+}^{1-\alpha, \sigma} \frac{\partial \xi}{\partial s}(0,u(t),t) \frac{\partial L}{\partial y} - \mathbb{I}_{a^+}^{1-\alpha, \sigma} \frac{\partial \xi}{\partial s}(0,u(t),t) \frac{\rm d}{{\rm d}t}\frac{\partial L}{\partial y} \right){\rm d}t\\
	&= \int_{a}^{b} \sigma \mathbb{I}_{a^+}^{1-\alpha, \sigma} \frac{\partial \xi}{\partial s}(0,u(t),t) \frac{\partial L}{\partial y}{\rm d}t\\& - \left(\lim_{t\to b^-} \left(\frac{\partial L}{\partial y}\mathbb{I}_{a^+}^{1-\alpha, \sigma}\frac{\partial \xi}{\partial s}(0,u(t),t)\right) - \lim_{t\to a^+}\left(\frac{\partial L}{\partial y}\mathbb{I}_{a^+}^{1-\alpha, \sigma}\frac{\partial \xi}{\partial s}(0,u(t),t)\right) - \int_{a}^{b}\frac{\partial L}{\partial y} \frac{\rm d}{{\rm d}t} \mathbb{I}_{a^+}^{1-\alpha, \sigma}\frac{\partial \xi}{\partial s}(0,u(t),t) {\rm d}t \right)\\
	&= \int_{a}^{b} \frac{\partial L}{\partial y}\left(\frac{\rm d}{{\rm d}t} + \sigma \right) \mathbb{I}_{a^+}^{1-\alpha, \sigma} \frac{\partial \xi}{\partial s}(0,u(t),t) {\rm d} x\\&- \lim_{t\to b^-}\left(\frac{\partial L}{\partial y}\mathbb{I}_{a^+}^{1-\alpha, \sigma}\frac{\partial \xi}{\partial s}(0,u(t),t)\right) + \lim_{t\to a^+}\left(\frac{\partial L}{\partial y}\mathbb{I}_{a^+}^{1-\alpha, \sigma}\frac{\partial \xi}{\partial s}(0,u(t),t)\right)\\
	&= \int_{a}^{b} \frac{\partial L}{\partial y}{^{C}}\mathbb{D}_{a^+}^{\alpha, \sigma}\frac{\partial \xi}{\partial s}(0,u(t),t) {\rm d}t+\frac{\frac{\partial \xi}{\partial s}(0,u(a),a)}{\Gamma (1-\alpha)}\int_{a}^{b}(t-a)^{-\alpha}e^{-\sigma (t-a)} \frac{\partial L}{\partial y}{\rm d}t \\&- \lim_{t\to b^-}\left(\frac{\partial L}{\partial y}\mathbb{I}_{a^+}^{1-\alpha, \sigma}\frac{\partial \xi}{\partial s}(0,u(t),t)\right) + \lim_{t\to a^+}\left(\frac{\partial L}{\partial y}\mathbb{I}_{a^+}^{1-\alpha, \sigma}\frac{\partial \xi}{\partial s}(0,u(t),t)\right).
\end{aligned}\label{TF05}
\end{equation}
In the same way, we can show that 
\begin{equation}\label{TF06}
	\begin{aligned}
		 &\int_{a}^{b} \left(\sigma  \frac{\partial \xi}{\partial s}(0,u(t),t)\mathbb{I}_{b^-}^{1-\alpha, \sigma}\frac{\partial L}{\partial y} - \frac{\rm d}{{\rm d}t} \frac{\partial \xi}{\partial s}(0,u(t),t) \mathbb{I}_{b^-}^{1-\alpha, \sigma}\frac{\partial L}{\partial y} \right) {\rm d}t\\
		&= - \int_{a}^{b} \frac{\partial \xi}{\partial s}(0,u(t),t) \mathbb{D}_{b^-}^{\alpha, \sigma}\frac{\partial L}{\partial y}{\rm d}t - \lim_{t\to b^-}\left(\frac{\partial \xi}{\partial s}(0,u(t),t)\mathbb{I}_{b^-}^{1-\alpha, \sigma} \frac{\partial L}{\partial y}\right) + \lim_{t\to a^+}\left( \frac{\partial \xi}{\partial s}(0,u(t),t) \mathbb{I}_{b^-}^{1-\alpha, \sigma}\frac{\partial L}{\partial y}\right).
	\end{aligned}
\end{equation}
Finally, combining \eqref{eq:cnsidf11}, \eqref{TF10}--\eqref{TF06}, and taking into account that 
\begin{multline*}
 \lim_{t\to b^-}\left(\frac{\partial L}{\partial y}\mathbb{I}_{a^+}^{1-\alpha, \sigma}\frac{\partial \xi}{\partial s}(0,u(t),t)\right) - \lim_{t\to a^+}\left(\frac{\partial L}{\partial y}\mathbb{I}_{a^+}^{1-\alpha, \sigma}\frac{\partial \xi}{\partial s}(0,u(t),t)\right)\\ +\lim_{t\to b^-}\left(\frac{\partial \xi}{\partial s}(0,u(t),t)\mathbb{I}_{b^-}^{1-\alpha, \sigma} \frac{\partial L}{\partial y}\right) -\lim_{t\to a^+}\left( \frac{\partial \xi}{\partial s}(0,u(t),t) \mathbb{I}_{b^-}^{1-\alpha, \sigma}\frac{\partial L}{\partial y}\right)=\textnormal{constant},
\end{multline*}
we can conclude that
\begin{multline*}
 \frac{\partial L}{\partial y}\mathbb{I}_{a^+}^{1-\alpha, \sigma}\frac{\partial \xi}{\partial s}(0,u(t),t) + \frac{\partial \xi}{\partial s}(0,u(t),t) \mathbb{I}_{b^-}^{1-\alpha, \sigma}\frac{\partial L}{\partial y}\\
 =\lim_{t\to b^-}\left(\frac{\partial L}{\partial y}\mathbb{I}_{a^+}^{1-\alpha, \sigma}\frac{\partial \xi}{\partial s}(0,u(t),t)\right)  -\lim_{t\to a^+}\left( \frac{\partial \xi}{\partial s}(0,u(t),t) \mathbb{I}_{b^-}^{1-\alpha, \sigma}\frac{\partial L}{\partial y}\right)
\end{multline*}
is a constat of a motion in the sense of Definition~\ref{lc1}. The prove is complete.
\end{proof}

\subsection{Comments}
\begin{itemize}
	\item[(i)] \emph{About Theorem~\ref{CNT}}
	
	Theorem~\ref{CNT} is quite interesting because it provides a very explicit expression of a constant of motion in terms of symmetry group and Lagrangian for fractional problems of calculus of variations. Several authors have already proved Lemma~\ref{eq:cnsidf11}, but in our opinion, none of them obtained such an explicit expression of a constant of motion as in \eqref{NCL} \cite{Atana, bourdin1,FL}.

	\item[(ii)]  \emph{Link with the fractional problem of the calculus of variations  in the sense of the usual Caputo derivative ($\sigma=0$) \cite{FL}}.
	
	In this case,  Problem~\ref{pb} is reduced to the
	fractional problem of the calculus of variations,
	\begin{equation}
		\label{eq:pbcv11}
		\mathcal{J}[(\cdot)] = \int_a^b
		L\left(u(t),{^{C}}\mathbb{D}_{a^+}^{\alpha, 0}u(t),t\right) \longrightarrow \min ,
	\end{equation}
	and Theorem~\ref{CNT} asserts that the quantity 
	\begin{equation}\label{GM}
		\frac{\partial L}{\partial y}\left(u(t),{^{C}}\mathbb{D}_{a^+}^{\alpha, 0}u(t),t\right) \mathbb{I}_{a^+}^{1-\alpha, 0}\frac{\partial \xi}{\partial s}(0,u(t),t) + \frac{\partial \xi}{\partial s}(0,u(t),t) \mathbb{I}_{b^-}^{1-\alpha, 0}\frac{\partial L}{\partial y}\left(u(t),{^{C}}\mathbb{D}_{a^+}^{\alpha, 0}u(t),t\right) 
	\end{equation}
	is a constant of motion on $X'$ and $S'$. Note that, $\mathbb{I}_{a^+}^{1-\alpha, 0}$ and $\mathbb{I}_{b^-}^{1-\alpha, 0}$ denote the left and right Riemann-Liouville fractional integrals,
	respectively.

	\item[(iii)]\emph{Link with the classical problem of the calculus of variations}
	
	In case $\alpha=1$ and $\sigma=0$, Problem~\ref{pb} is reduced to the
	classical problem of the calculus of variations,
	\begin{equation}
		\label{eq:pbcv}
		I[(\cdot)] = \int_a^b
		L\left(u(t),\frac{{\rm d}u}{{\rm d}t}(t),t\right) \longrightarrow \min \, ,
	\end{equation}
	and one obtains from Theorem~\ref{CNT} the standard
	Noether's theorem \cite{Noether:1918}:
	\begin{equation}
		\label{eq:classNoetCL}
		C[u]_t=
		\frac{\partial L}{\partial y}\left(u(t),\frac{{\rm d}u}{{\rm d}t}(t),t\right)\frac{\partial \xi}{\partial s}(0,u(t),t)
	\end{equation}
	is a constant of motion, \textrm{i.e.} \eqref{eq:classNoetCL} is
	constant along all  solutions of the Euler-Lagrange equation
	\begin{equation}
		\label{eq:EL}
		\frac{\partial L}{\partial x}\left(u(t),\frac{{\rm d}u}{{\rm d}t}(t),t\right)=\frac{\rm d}{{\rm d}t}\frac{\partial L}{\partial y}\left(u(t),\frac{{\rm d}u}{{\rm d}t}(t),t\right)
	\end{equation}
	(this classical equation is obtained from \eqref{EL}
	putting $\alpha=1$ and $\sigma=0$).
	\item[(iv)] \emph{Coherence's principle}
	
	Finally, we have an important conclusion: the \emph{principle of coherence} is respected,  \textnormal{i.e.}, the following diagram commutes

\begin{equation}\label{diagram}
	\begin{tikzcd}
			\textnormal{Problem~\ref{pb}} \arrow[r, "\sigma=0"] \arrow[d] & \textnormal{Problem~\eqref{eq:pbcv11}} \arrow[r, "\alpha=1"] \arrow[d] \arrow[dr, phantom, ] & \textnormal{Problem~\eqref{eq:pbcv}} \arrow[d] \\
		\textnormal{Theorem~\ref{CNT}}  \arrow[r, "\sigma=0"']          & \textnormal{Constant of motion~\eqref{GM}} \arrow[r, "\alpha=1"']                                         & \textnormal{Constant of motion~\eqref{eq:classNoetCL}}. 
	\end{tikzcd}
\end{equation}

\end{itemize}

\subsection{Illustrative applications}

\subsubsection{Application to Problem~\ref{pb} when the Lagragian $L$ does not depend explicitly on $u$} --- In classical mechanics, when problem \eqref{eq:pbcv} does not
depend explicitly on $u$, \textrm{i.e.}, $L \equiv
L\left(\dot{u}(t),t\right)$, it follows from 
\eqref{eq:EL} that the generalized momentum $p=\frac{\partial L}{\partial y}$ is a constant of motion. This is also an immediate consequence of Noether's theorem
\cite{Noether:1918}: from the invariance with respect to
translations on $u$ $\left( \xi(s,u(t),t)= u(t)+s\right)$, it follows from
\eqref{eq:classNoetCL} that $p$ is a constant of motion.  In this case,  our variational problem reduces to

\begin{equation}
	\label{eq:pbcv2}
	\mathcal{J}[(\cdot)] = \int_a^b
	L\left({^{C}}\mathbb{D}_{a^+}^{\alpha, \sigma}u(t),t\right) \longrightarrow \min .
\end{equation}
Furthermore, in such circumstances, Theorem~\ref{CNT} provides a new interesting insight for problem \eqref{eq:pbcv2}:

\begin{corollary}
\label{cor:FOCP:CV}
For the fractional problem \eqref{eq:pbcv2}, the quantity
\begin{equation}
\lim_{t\to b^-}\left(\frac{\partial L}{\partial y}[u]_t\mathbb{I}_{a^+}^{1-\alpha, \sigma}e^{-\sigma t}\right) -\lim_{t\to a^+}\left( e^{-\sigma t} \mathbb{I}_{b^-}^{1-\alpha, \sigma}\frac{\partial L}{\partial y}[u]_t
\right)\end{equation}
is a constant of motion in the sense of Definition~\ref{lc1}.
\end{corollary}
\begin{proof}
Provided that the Lagrangian does not depend explicitly on $u$,  it is easy to check that invariance condition \eqref{invg} is satisfied with $\xi(s,u(t),t)=u(t) + se^{-\sigma t}$. In fact, \eqref{invg} holds trivially proving that ${^{C}}\mathbb{D}_{a^+}^{\alpha, \sigma}\xi\left(s,u(t),t\right)={^{C}}\mathbb{D}_{a^+}^{\alpha, \sigma}u(t)$:
\begin{equation*}
	\begin{aligned}
{^{C}}\mathbb{D}_{a^+}^{\alpha, \sigma}\xi(s,u(t),t) & = \frac{e^{-\sigma t}}
{\Gamma (1-\alpha)} \int_{a}^{t} (t-\theta)^{-\alpha}\frac{{\rm d}}{{\rm d}\theta}\left[e^{\sigma \theta}\xi(s,u(\theta),\theta)\right]{\rm d}\theta\\
& =	\frac{e^{-\sigma t}}
{\Gamma (1-\alpha)} \int_{a}^{t} (t-\theta)^{-\alpha}\frac{{\rm d}}{{\rm d}\theta}\left[e^{\sigma \theta}\left(u(t) + se^{-\sigma \theta}\right)\right]{\rm d}\theta	\\
&=\frac{e^{-\sigma t}}
{\Gamma (1-\alpha)} \int_{a}^{t} (t-\theta)^{-\alpha}\frac{{\rm d}}{{\rm d}\theta}\left[e^{\sigma \theta}u(\theta)\right]{\rm d}\theta\\
&={^{C}}\mathbb{D}_{a^+}^{\alpha, \sigma}u(t).	
\end{aligned}
\end{equation*}
Using the previous notations, one has $\frac{\partial \xi}{\partial s}(0,u(t),t)=e^{-\sigma t}$. Desired conclusion follows from Theorem~\ref{CNT}.
\end{proof}

We can easily check the community of the diagram \eqref{diagram} through the Corollary~\ref{cor:FOCP:CV}, \textnormal{i.e.}, when $\alpha=1$, $\sigma=0$ and  $L \equiv
L\left(\dot{u}(t),t\right)$, the generalized momentum $p=\frac{\partial L}{\partial y}$ is a constant of motion along all the solutions of the Euler-Lagrange equation \eqref{eq:EL}.

\subsubsection{Example}---
In the following example we investigate an existence result of the following linear boundary value problem
\begin{equation}\label{BVP01}
\begin{aligned}
&\mathbb{D}_{b-}^{\alpha, \sigma}(\mathbb{D}_{a^+}^{\alpha, \sigma}u(x)) + u(x)= f(x),\;\;\mbox{a.e. }x\in [a,b]\\
&u(a) = u(b) = 0,
\end{aligned}
\end{equation}
where $\alpha \in (\frac{1}{2}, 1)$ and $f\in L^2[a,b]$. To solve this problem we use the classical Lax-Milgran theorem.

\noindent 
Now, for $\alpha \in (\frac{1}{2}, 1)$ and $\sigma >0$, we claim that problem (\ref{BVP01}) has a solution $u\in \mathbb{H}_{0}^{\alpha, \sigma}(a,b)$. In fact. Let us consider the continuous coercive bilinear form 
$$
\begin{aligned}
a: \mathbb{H}_{0}^{\alpha, \sigma}\times \mathbb{H}_{0}^{\alpha, \sigma}(a,b) &\to \R\\
(u_1, u_2) &\to a(u_1, u_2) = \int_{a}^{b} \mathbb{D}_{a^+}^{\alpha, \sigma}u_1(x) \mathbb{D}_{a^+}^{\alpha, \sigma} u_2(x)dx + \int_{a}^{b}u(x)v(x)dx,
\end{aligned}
$$ 
the linear form 
$$
\begin{aligned}
\phi: \mathbb{H}_{0}^{\alpha, \sigma}(a,b) &\to \R\\
u &\to \phi (u) = \int_{a}^{b} f(x)u(x)dx.
\end{aligned}
$$
Then, by Lax-Milgran theorem there is a unique $u\in \mathbb{H}_{0}^{\alpha, \sigma}(a,b)$ such that 
\begin{equation}\label{BVP02}
\int_{a}^{b} \mathbb{D}_{a^+}^{\alpha, \sigma} u(x) \mathbb{D}_{a^+}^{\alpha, \sigma} v(x)dx + \int_{a}^{b}u(x)v(x)dx = \int_{a}^{b} f(x)v(x)dx
\end{equation}
for any $v\in \mathbb{H}_{0}^{\alpha, \sigma}(a,b)$.

\section{weak solution trough Mountain pass theorem }
In this section we are going to use the well known mountain pass theorem due to Ambrosetti and Rabinowitz \cite{AmRa} to get another type of weak solution for (\ref{EL}). That is, we are going to use the following theorem 
\begin{theorem}\label{MPT}
\cite{AmRa} Let $X$ be a real Banach space and $I\in C^1(X, \R)$ satisfies the following conditions:
\begin{enumerate}
\item $I$ satisfies Palais-Smale condition, that is, if any sequence $(u_n)_{n\in \N} \subset X$, for which $(I(u_n))_{n\in \N}$ is bounded and $I'(u_n) \to 0$ as $n\to \infty$, possesses a convergent subsequence in $X$,
\item $I(0) = 0$,
\item There are $\rho>0$ and $\eta >0$ such that $I\Big|_{\partial B(0, \rho)}\geq \eta$,
\item There is $e\in X$ with $\|e\|_X>\rho$ such that $I(e) <0$.
\end{enumerate}
Then $I$ possesses a critical value $c\geq \eta$ given by 
$$
c = \inf_{g\in \Gamma}\max_{t\in [0,1]}I(g(t)),
$$
where $\Gamma = \{g\in C([0,1], X): g(0) = 0,\;\;g(1) = e\}$.
\end{theorem}

In order to apply Theorem \ref{MPT}, we suppose that the Lagrangian $L: \R\times \R \times [a,b]\to \R$ is of $C^1$ class and satisfies the following condition 
\begin{enumerate}
\item[$(M_1)$] $L(x,\cdot, t)$ is convex for all $(x,t)\in \R\times[a,b]$.
\item[$(M_2)$] There are $\rho\in C(\R^+, \R^+)$ and $\vartheta\in L^1([a,b], \R^+)$ such that for all $(x,y,t)\in \R\times \R \times [a,b]$
$$
\begin{aligned}
&|L(x,y,t)| \leq \rho(|x|)(\vartheta(t) + |y|^2),\\
&\left|\frac{\partial L}{\partial x}(x,y,t)\right|\leq \rho(|x|)(\vartheta(t) + |y|^2),\;\;\mbox{and}\\
&\left|\frac{\partial L}{\partial y}(x,y,t) \right|^2 \leq \rho(|x|)(\vartheta(t) + |y|^2)
\end{aligned}
$$
\item[$(M_3)$] There is $\mu_L\in (0,2)$ such that $\forall (x,y,t)\in \R\times \R\times[a,b]$
$$
L_x(x,y,t) x + L_y(x,y,t)y\leq \mu_LL(x,y,t)
$$
\item[$(M_4)$] There is $\Lambda>0$ such that for all $(x,y,t)\in \R\times \R \times[a,b]$ 
$$
L(x,y,t) \geq \Lambda |y|^2.
$$
\item[$(M_5)$] $L(x,0,t) = 0$ for all $(x,t)\in \R\times [a,b]$.
\end{enumerate}

As a consequence of condition $(M_3)$ we have the following result. 
\begin{lemma}\label{Mlm1}
Suppose that $(M_3)$ holds. Then, for every $\lambda >1$
$$
L(\lambda x, \lambda y, t) \leq \lambda^{\mu_L}L(x,y,t)\quad \mbox{for all $(x,y,t) \in \R\times \R \times[a,b]$}.
$$
\end{lemma}

\noindent 
In order to apply Theorem \ref{MPT} we are going to show that $\mathcal{J}\in C^1(\mathbb{H}_{0}^{\alpha, \sigma}(a,b), \R)$.
\begin{lemma}\label{Dife1}
If $u\in \mathbb{H}_{0}^{\alpha, \sigma}(a,b)$, then $\frac{\partial L}{\partial x}(x,y,\cdot) \in L^1$ and $\frac{\partial L}{\partial y}(x,y, \cdot)\in L^2$
\end{lemma}
\begin{proof}
Define the non-decreasing function 
$$
\xi (s) = \sup_{\tau \in [0,s]}\rho(\tau).
$$ 
Then, for any $u\in \mathbb{H}_{0}^{\alpha, \sigma}(a,b)$ and Theorem \ref{embedding} we have 
\begin{equation}\label{dif01}
\rho(|u(t)|) \leq \xi (\|u\|_\infty) \leq \xi \left(\frac{\sqrt{\gamma (2\alpha-1, 2\sigma (b-a))}}{(2\sigma)^{\alpha-\frac{1}{2}}\Gamma (\alpha)}\|u\| \right).
\end{equation}
Next, by $(M_2)$
$$
\begin{aligned}
\int_{a}^{b}\left|\frac{\partial L}{\partial x}(u, {^{C}}\mathbb{D}_{a^+}^{\alpha, \sigma}, t) \right|dt 
&\leq \xi \left(\frac{\sqrt{\gamma (2\alpha-1. 2\sigma(b-a))}}{(2\sigma)^{\alpha-\frac{1}{2}}\Gamma (\alpha)} \right) \int_{a}^{b}\Big(\vartheta(t) + |{^{C}}\mathbb{D}_{a^+}^{\alpha, \sigma}u(t)|^2 \Big)dt < \infty.
\end{aligned}
$$
and 
$$
\int_{a}^{b}\left|\frac{\partial L}{\partial y}(u, {^{C}}\mathbb{D}_{a^+}^{\alpha, \sigma} u, t) \right|^2dt \leq \xi \left(\frac{\sqrt{\gamma (2\alpha-1, 2\sigma(b-a))}}{(2\sigma)^{\alpha-\frac{1}{2}}\Gamma (\alpha)} \right) \int_{a}^{b}\Big(\vartheta(t) + |{^{C}}\mathbb{D}_{a^+}^{\alpha, \sigma}u(t)|^2 \Big)dt < \infty.
$$
\end{proof}

\begin{theorem}\label{Dife2}
$\mathcal{J} \in C^1(\mathbb{H}_{0}^{\alpha, \sigma}(a,b), \R)$. Moreover 
$$
\mathcal{J}'(u)\varphi = \int_{a}^{b} \left( \frac{\partial L}{\partial x}(u, {^{C}}\mathbb{D}_{a^+}^{\alpha, \sigma}u, t)\varphi + \frac{\partial L}{\partial y}(u, {^{C}}\mathbb{D}_{a^+}^{\alpha, \sigma} u, t){^{C}}\mathbb{D}_{a^+}^{\alpha, \sigma} \varphi \right)dt
$$
\end{theorem}
\begin{proof}
$(M_2)$ and (\ref{dif01}) yield that 
$$
|\mathcal{J}(u)| 
\leq \xi \left(\frac{\sqrt{\gamma (2\alpha - 1, 2\sigma (b-a))}}{(2\sigma)^{\alpha-\frac{1}{2}}\Gamma (\alpha)} \right) \int_{a}^{b}\Big(\vartheta(t) + |{^{C}}\mathbb{D}_{a^+}^{\alpha, \sigma}u|^2 \Big)dt <\infty.
$$
For $u\in \mathbb{H}_{0}^{\alpha, \sigma}(a,b)$, $\varphi \in \mathbb{H}_{0}^{\alpha, \sigma}(a,b)\setminus \{0\}$ and $t\in (a,b)$, $s\in [-1,1]$, let us define  
$$
H(s,t) := L\Big(u+s\varphi, {^{C}}\mathbb{D}_{a^+}^{\alpha, \sigma}u + s{^{C}}\mathbb{D}_{a^+}^{\alpha, \sigma} \varphi, t\Big).
$$
Hence, by $(M_2)$ we obtain 
$$
\begin{aligned}
&\int_{a}^{b}\left|\frac{\partial L}{\partial x} (u, {^{C}}\mathbb{D}_{a^+}^{\alpha, \sigma}u, t) \varphi \right|dt \leq \int_{a}^{b} \rho(|u+s\varphi|)(\vartheta(t) + |{^{C}}\mathbb{D}_{a^+}^{\alpha, \sigma}u|^2)|\varphi|dt\\
&\leq \xi \left(\frac{\sqrt{\gamma (2\alpha-1, 2\sigma (b-a))}}{(2\sigma)^{\alpha-\frac{1}{2}}\Gamma (\alpha)}\|u+s\varphi\| \right)\int_{a}^{b} \Big(\vartheta(t) + |{^{C}}\mathbb{D}_{a^+}^{\alpha, \sigma} u(t)|^2 \Big)|\varphi|dt < \infty.
\end{aligned}
$$
By H\"older inequality and Lemma \ref{Dife1} we obtain  
$$
\begin{aligned}
\int_{a}^{b} \left|\frac{\partial L}{\partial y}(u, {^{C}}\mathbb{D}_{a^+}^{\alpha, \sigma}u, t){^{C}}\mathbb{D}_{a^+}^{\alpha, \sigma} \varphi \right|dt &\leq \int_{a}^{b} &\leq \left(\int_{a}^{b}\left|\frac{\partial L}{\partial y}(u, {^{C}}\mathbb{D}_{a^+}^{\alpha, \sigma}u, t) \right|^2dt \right)^{1/2}\left(\int_{a}^{b}|{^{C}}\mathbb{D}_{a^+}^{\alpha, \sigma}\varphi|^2dt \right)^{1/2}  < \infty,
\end{aligned}
$$
and then 
$$
\begin{aligned}
&\int_{a}^{b}\left|\frac{\partial H}{\partial s}(s,t) \right|dt \\
&= \int_{a}^{b}\left| \frac{\partial L}{\partial x}\Big(u + s\varphi, {^{C}}\mathbb{D}_{a^+}^{\alpha, \sigma}u + s {^{C}}\mathbb{D}_{a^+}^{\alpha, \sigma}\varphi, t\Big)\varphi + \frac{\partial L}{\partial y}\Big(u+s\varphi, {^{C}}\mathbb{D}_{a^+}^{\alpha, \sigma}u + s{^{C}}\mathbb{D}_{a^+}^{\alpha, \sigma}\varphi, t\Big){^{C}}\mathbb{D}_{a^+}^{\alpha, \sigma}\varphi\right|dt < \infty.
\end{aligned}
$$ 
Consequently, $\mathcal{J}$ has a directional derivative with
$$
\mathcal{J}'(u)\varphi = \frac{d}{ds}\mathcal{J}(u+s\varphi)\Big|_{s=0} = \int_{a}^{b} \left(\frac{\partial L}{\partial x}(u,{^{C}}\mathbb{D}_{a^+}^{\alpha, \sigma}u, t)\varphi + \frac{\partial L}{\partial y}(u, {^{C}}\mathbb{D}_{a^+}^{\alpha, \sigma}u, t){^{C}}\mathbb{D}_{a^+}^{\alpha, \sigma} \varphi \right)dt
$$ 
Lemma \ref{Dife1}, H\"older inequality and (\ref{conti}) yield that 
$$
\begin{aligned}
|\mathcal{J}'(u)\varphi| &\leq \left\|\frac{\partial L}{\partial x}(u,{^{C}}\mathbb{D}_{a^+}^{\alpha, \sigma}u, \cdot) \right\|_{L^1}\|\varphi\|_\infty + \left\| \frac{\partial L}{\partial y}(u, {^{C}}\mathbb{D}_{a^+}^{\alpha, \sigma}u, \cdot) \right\|_{L^2}\|{^{C}}\mathbb{D}_{a^+}^{\alpha, \sigma} \varphi\|_{L^2}\\
&\leq \left( \frac{\sqrt{\gamma (2\alpha-1, 2\sigma (b-a))}}{(2\sigma)^{\alpha-\frac{1}{2}}\Gamma (\alpha)}\left\|\frac{\partial L}{\partial x}(u, {^{C}}\mathbb{D}_{a^+}^{\alpha, \sigma}u, \cdot) \right\|_{L^1} + \left\| \frac{\partial L}{\partial y}(u, {^{C}}\mathbb{D}_{a^+}^{\alpha, \sigma}u, \cdot)\right\|_{L^2} \right)\|\varphi\|.
\end{aligned}
$$
Now we claim that $\mathcal{J'}$ is continuous. In fact, let $u_n \to u$ in $\mathbb{H}_{0}^{\alpha, \sigma}(a,b)$. So
$$
u_n \to u\;\;\mbox{in $L^2(a,b)$},\quad {^{C}}\mathbb{D}_{a^+}^{\alpha, \sigma}u_n\to {^{C}}\mathbb{D}_{a^+}^{\alpha, \sigma}u \;\;\mbox{in $L^2(a,b)$} 
$$
and there is $M>0$ such that 
\begin{equation}\label{bb}
\|u_n\| < M.
\end{equation}
Hence, up to a subsequence there is $h\in L^1(a,b)$ such that 
$$
{^{C}}\mathbb{D}_{a^+}^{\alpha, \sigma} u_n(t) \to {^{C}}\mathbb{D}_{a^+}^{\alpha, \sigma} u(t)\;\;\mbox{a.e. and}\quad |{^{C}}\mathbb{D}_{a^+}^{\alpha, \sigma}u_n|^2\leq h. 
$$
So, by $(M_2)$ and (\ref{bb}) we obtain  
$$
\left|\frac{\partial L}{\partial x}(u_n, {^{C}}\mathbb{D}_{a^+}^{\alpha, \sigma}u_n, t) \right| \leq \xi (\|u_n\|)(\vartheta(t) + |{^{C}}\mathbb{D}_{a^+}^{\alpha, \sigma}u_n|^2) \leq \xi (M)(\vartheta(t) + h(t)).
$$
Moreover, since $L\in C^1$ we have 
$$
\frac{\partial L}{\partial x}(u_n(t), {^{C}}\mathbb{D}_{a^+}^{\alpha, \sigma} u_n(t), t) \to \frac{\partial L}{\partial x} (u(t), {^{C}}\mathbb{D}_{a^+}^{\alpha, \sigma}u(t), t)\;\;\mbox{for a.e. $t\in (a,b)$}.
$$
Applying Lebesgue's dominated convergence theorem we obtain 
$$
\int_{a}^{b} \frac{\partial L}{\partial x}(u_n, {^{C}}\mathbb{D}_{a^+}^{\alpha, \sigma}u_n, t)\varphi dt \to \int_{a}^{b} \frac{\partial L}{\partial x}(u, {^{C}}\mathbb{D}_{a^+}^{\alpha, \sigma}u, t)\varphi dt.
$$
In the same way we can show that 
$$
\int_{a}^{b} \frac{\partial L}{\partial y}(u_n, {^{C}}\mathbb{D}_{a^+}^{\alpha, \sigma} u_n, t){^{C}}\mathbb{D}_{a^+}^{\alpha, \sigma}\varphi dt \to \int_{a}^{b}\frac{\partial L}{\partial y}(u, {^{C}}\mathbb{D}_{a^+}^{\alpha, \sigma} u ,t) {^{C}}\mathbb{D}_{a^+}^{\alpha, \sigma}\varphi dt.
$$
The proof is finished.
\end{proof}

Now, by following some ideas found in \cite{chma}, we are going to verify the Palais-Smale condition
\begin{lemma}\label{Mlm2}
The functional $\mathcal{J}$ satisfies the Palais-Smale condition
\end{lemma}
\begin{proof}
Fix $u\in \mathbb{H}_{0}^{\alpha, \sigma}(a,b)$. From $(M_3)$ and $(M_4)$ we obtain 
$$
\begin{aligned}
&\int_{a}^{b}2\mu_LL(u,{^{C}}\mathbb{D}_{a^+}^{\alpha, \sigma} u, t)dt -\int_{a}^{b}\left( \frac{\partial L}{\partial x}(u, {^{C}}\mathbb{D}_{a^+}^{\alpha, \sigma}u, t)u + \frac{\partial L}{\partial y}(u, {^{C}}\mathbb{D}_{a^+}^{\alpha, \sigma}u, t){^{C}}\mathbb{D}_{a^+}^{\alpha, \sigma}u\right) dt\\
&\geq \mu_L\int_{a}^{b}L(u, {^{C}}\mathbb{D}_{a^+}^{\alpha, \sigma}u, t)dt \geq \Lambda \mu_L \int_{a}^{b}|{^{C}}\mathbb{D}_{a^+}^{\alpha, \sigma} u|^2dt. 
\end{aligned}
$$
Hence 
\begin{equation}\label{M01}
\Lambda \mu_L\|u\|^2 \leq 2\theta_L |\mathcal{J}(u)| + \|\mathcal{J}'(u)\|\|u\|.
\end{equation}
Let $(u_n)_{n\in \N} \subset \mathbb{H}_{0}^{\alpha, \sigma}(a,b)$ be a Palais-Smale sequence, i.e. $(\mathcal{J}(u_n))_{n\in \N}$ is bounded and $\mathcal{J}'(u_n)\to 0$ as $n\to \infty$. By (\ref{M01}), $(u_n)_{n\in \N}$ is bounded in $\mathbb{H}_{0}^{\alpha, \sigma}(a,b)$, so up to a subsequence we have 
\begin{equation}\label{M02}
u_n \rightharpoonup u\;\;\mbox{in $\mathbb{H}_{0}^{\alpha, \sigma}(a,b)$}\quad \mbox{and}\quad u_n \to u\;\;\mbox{in $C\overline{(a,b)}$}.
\end{equation}
So
$$
\begin{aligned}
&\lim_{n\to \infty} \mathcal{J}'(u_n)(u_n-u) \\
&= \lim_{n\to \infty} \int_{a}^{b}\left[\frac{\partial L}{\partial x}(u_n, {^{C}}\mathbb{D}_{a^+}^{\alpha, \sigma} u_n, t)(u_n-u) + \frac{\partial L}{\partial y}(u_n, {^{C}}\mathbb{D}_{a^+}^{\alpha, \sigma}u_n, t)({^{C}}\mathbb{D}_{a^+}^{\alpha, \sigma}u_n - {^{C}}\mathbb{D}_{a^+}^{\alpha, \sigma}u) \right]dt = 0.
\end{aligned}
$$
Since $(u_n)_{n\in \N}$ is bounded in $\mathbb{H}_{0}^{\alpha, \sigma}(a,b)$, there exists $K>0$ such that 
$$
\rho(|u_n(t)|) \leq \xi(\|u_n\|_\infty) \leq K.
$$
Thus, by $(M_2)$ we get 
$$
\begin{aligned}
\int_{a}^{b}\left|\frac{\partial L}{\partial x}(u_n, {^{C}}\mathbb{D}_{a^+}^{\alpha, \sigma} u_n, t) \right|dt \leq K(\|\vartheta\|_{L^1} + \|u_n\|)
\end{aligned}
$$
which implies that $\|\frac{\partial L}{\partial x}(u_n, {^{C}}\mathbb{D}_{a^+}^{\alpha, \sigma}u_n, \cdot)\|_{L^1}$ is uniformly bounded, hence, by (\ref{M02}) we get 
$$
\Big|\int_{a}^{b} \frac{\partial L}{\partial x}(u_n, {^{C}}\mathbb{D}_{a^+}^{\alpha, \sigma}u_n, t)(u_n-u)dt \Big| \leq \Big\|\frac{\partial L}{\partial x}L(u_n, {^{C}}\mathbb{D}_{a^+}^{\alpha, \sigma}u_n, \cdot) \Big\|_{L^1} \|u_n - u\|_{\infty} \to 0
$$
and consequently 
\begin{equation}\label{M03}
\lim_{n\to \infty} \int_{a}^{b} \frac{\partial L}{\partial y}(u_n, {^{C}}\mathbb{D}_{a^+}^{\alpha, \sigma}u_n, t)({^{C}}\mathbb{D}_{a^+}^{\alpha, \sigma} u_n - {^{C}}\mathbb{D}_{a^+}^{\alpha, \sigma} u)dt =0.
\end{equation}
By continuity of $L$ we have that 
$$
L(u_n(t), \pm {^{C}}\mathbb{D}_{a^+}^{\alpha, \sigma} u(t), t) \to L(u(t), \pm {^{C}}\mathbb{D}_{a^+}^{\alpha, \sigma}u(t), t)\;\;\mbox{a.e.}
$$
Furthermore, by $(M_2)$ and ${^{C}}\mathbb{D}_{a^+}^{\alpha, \sigma} u\in L^2(a,b)$ we get 
$$
|L(u_n(t), {^{C}}\mathbb{D}_{a^+}^{\alpha, \sigma} u(t), t)| \leq K(\vartheta(t) + |{^{C}}\mathbb{D}_{a^+}^{\alpha, \sigma} u|^2) \in L^1.
$$
Hence, by the Lebesgue's dominated convergence theorem 
\begin{equation}\label{M04}
\lim_{n\to \infty}\int_{a}^{b}L(u_n(t), \pm {^{C}}\mathbb{D}_{a^+}^{\alpha, \sigma} u(t), t)dt = \int_{a}^{b} L(u(t), \pm {^{C}}\mathbb{D}_{a^+}^{\alpha, \sigma}u(t), t)dt
\end{equation}
On the other hand, $(M_1)$, (\ref{M03}) and (\ref{M04}) yield that 
$$
\begin{aligned}
&\limsup_{n\to \infty}\int_{a}^{b}L(u_n, {^{C}}\mathbb{D}_{a^+}^{\alpha, \sigma}u_n, t)dt\\
& \leq \lim_{n\to \infty}\int_{a}^{b} \left( L(u_n, {^{C}}\mathbb{D}_{a^+}^{\alpha, \sigma}u, t) + \frac{\partial L}{\partial y}(u_n, {^{C}}\mathbb{D}_{a^+}^{\alpha, \sigma}u_n, t)\Big({^{C}}\mathbb{D}_{a^+}^{\alpha, \sigma}u_n - {^{C}}\mathbb{D}_{a^+}^{\alpha, \sigma}u\Big) \right)dt= \int_{a}^{b}L(u, {^{C}}\mathbb{D}_{a^+}^{\alpha, \sigma}u,t)dt.
\end{aligned}
$$
Since $L(u_n(t), {^{C}}\mathbb{D}_{a^+}^{\alpha, \sigma} u_n(t), t)\geq 0$ and $L(u_n(t), {^{C}}\mathbb{D}_{a^+}^{\alpha, \sigma} u_n(t), t) \to L(u(t), {^{C}}\mathbb{D}_{a^+}^{\alpha, \sigma} u(t), t)$ a.e., by Fatou's lemma we get 
$$
\int_{a}^{b}L(u, {^{C}}\mathbb{D}_{a^+}^{\alpha, \sigma}u, t)dt \leq \lim_{n\to \infty}\int_{a}^{b} L(u_n, {^{C}}\mathbb{D}_{a^+}^{\alpha, \sigma} u_n, t)dt
$$
Consequently 
\begin{equation}\label{M05}
\lim_{n\to \infty}\int_{a}^{b}L(u_n, {^{C}}\mathbb{D}_{a^+}^{\alpha, \sigma}u_n, t)dt = \int_{a}^{b} L(u,{^{C}}\mathbb{D}_{a^+}^{\alpha, \sigma} u, t)dt.
\end{equation}
Now we are ready to show that $u_n \to u$ in $\mathbb{H}_{0}^{\alpha, \sigma}(a,b)$. Condition $(M_1)$ implies that 
$$
\frac{L(u_n(t), {^{C}}\mathbb{D}_{a^+}^{\alpha, \sigma} u_n(t), t) + L(u_n(t), -{^{C}}\mathbb{D}_{a^+}^{\alpha, \sigma} u(t), t)}{2} - L\left(u_n(t), \frac{{^{C}}\mathbb{D}_{a^+}^{\alpha, \sigma} u_n(t) - {^{C}}\mathbb{D}_{a^+}^{\alpha, \sigma}u(t)}{2}, t \right) \geq 0
$$
The continuity of $L$ combining with ${^{C}}\mathbb{D}_{a^+}^{\alpha, \sigma} u_n \to {^{C}}\mathbb{D}_{a^+}^{\alpha, \sigma} u$ a.e. and $(M_5)$ yield that 
$$
\begin{aligned}
&\lim_{n\to \infty} \frac{L(u_n(t), {^{C}}\mathbb{D}_{a^+}^{\alpha, \sigma} u_n(t), t) + L(u_n(t), -{^{C}}\mathbb{D}_{a^+}^{\alpha, \sigma}u(t), t)}{2} - L\left(u_n(t), \frac{{^{C}}\mathbb{D}_{a^+}^{\alpha, \sigma} u_n(t) - {^{C}}\mathbb{D}_{a^+}^{\alpha, \sigma}u(t)}{2}, t \right)\\
&= \frac{L(u(t, {^{C}}\mathbb{D}_{a^+}^{\alpha, \sigma} u(t), t) + L(u(t), -{^{C}}\mathbb{D}_{a^+}^{\alpha, \sigma}u(t), t))}{2}.
\end{aligned}
$$
Taking into account (\ref{M04}) and (\ref{M05}) we have
$$
\lim_{n\to \infty}\int_{a}^{b} \frac{L(u_n, {^{C}}\mathbb{D}_{a^+}^{\alpha, \sigma}u_n, t) + L(u_n, -{^{C}}\mathbb{D}_{a^+}^{\alpha, \sigma} u, t)}{2}dt = \int_{a}^{b} \frac{L(u, {^{C}}\mathbb{D}_{a^+}^{\alpha, \sigma} u, t) + L(u, -{^{C}}\mathbb{D}_{a^+}^{\alpha, \sigma} u, t)}{2}dt
$$ 
which implies 
$$
\begin{aligned}
&\int_{a}^{b} \frac{L(u,{^{C}}\mathbb{D}_{a^+}^{\alpha, \sigma}u, t) + L(u,-{^{C}}\mathbb{D}_{a^+}^{\alpha, \sigma}u, t)}{2} \\
&\leq \int_{a}^{b} \frac{L(u, {^{C}}\mathbb{D}_{a^+}^{\alpha, \sigma}u, t) + L(u, -{^{C}}\mathbb{D}_{a^+}^{\alpha, \sigma}u, t)}{2}dt - \limsup_{n\to \infty}\int_{a}^{b} L\left(u_n, \frac{{^{C}}\mathbb{D}_{a^+}^{\alpha, \sigma} u_n - {^{C}}\mathbb{D}_{a^+}^{\alpha, \sigma}u}{2}, t \right)dt.
\end{aligned}
$$
Thus
$$
\lim_{n\to \infty}\int_{a}^{b} L\left(u_n, \frac{{^{C}}\mathbb{D}_{a^+}^{\alpha, \sigma} u_n - {^{C}}\mathbb{D}_{a^+}^{\alpha, \sigma}u}{2}, t \right)dt = 0.
$$
Therefore, by $(M_4)$ we get 
$$
\lim_{n\to \infty}\int_{a}^{b}\left|\frac{{^{C}}\mathbb{D}_{a^+}^{\alpha, \sigma} u_n - {^{C}}\mathbb{D}_{a^+}^{\alpha, \sigma} u}{2} \right|^2dt \leq \lim_{n\to \infty} \frac{1}{\Lambda}\int_{a}^{b} L\left(u_n, \frac{{^{C}}\mathbb{D}_{a^+}^{\alpha, \sigma}u_n - {^{C}}\mathbb{D}_{a^+}^{\alpha, \sigma}u}{2}, t \right)dt.
$$ 
So $u_n \to u$ in $\mathbb{H}_{0}^{\alpha, \sigma}(a,b)$.

\end{proof}

\noindent
In what follows we are going to show $\it (2), (3)$ and $\it (4)$ of Theorem \ref{MPT}. By $(M_5)$ we have 
\begin{equation}\label{M06}
\mathcal{J}(0) = 0.
\end{equation}
\begin{lemma}\label{Mlm4}
There are $\varrho >0$ and $\eta >0$ such that 
$$
\inf_{\|u\|=\varrho}\mathcal{J}(u) \geq \eta.
$$
\end{lemma}
\begin{proof}
Let $\varrho = \frac{1}{{\Lambda}}$. Then, for $u\in \mathbb{H}_{0}^{\alpha, \sigma}(a,b)$ with $\|u\| = \varrho$, by $(M_4)$
$$
\mathcal{J}(u) \geq \Lambda \int_{a}^{b}|{^{C}}\mathbb{D}_{a^+}^{\alpha, \sigma}u|^2dt = \Lambda \varrho^2 = \frac{1}{\Lambda} = \eta>0.
$$
\end{proof}

\begin{lemma}\label{Mlm3}
There is $e\in \mathbb{H}_{0}^{\alpha, \sigma}(a,b)$ such that $\|e\| > \varrho$ and $\mathcal{J}(e) < 0$.
\end{lemma}
\begin{proof}
Choose $u_0\in \mathbb{H}_{0}^{\alpha, \sigma}(a,b)\setminus \{0\}$. Then, by Lemma \ref{Mlm1} and $(M_4)$, for any $\lambda >1$ we get
$$
\begin{aligned}
\mathcal{J}(\lambda u_0) &= 2\int_{a}^{b}L(\lambda u_0, \lambda {^{C}}\mathbb{D}_{a^+}^{\alpha, \sigma}u_0, t)dt - \int_{a}^{b}L(\lambda u_0, \lambda {^{C}}\mathbb{D}_{a^+}^{\alpha, \sigma} u_0, t)dt\\
&\leq \lambda^{\mu_L}\int_{a}^{b}L(u_0, {^{C}}\mathbb{D}_{a^+}^{\alpha, \sigma}u_0, t)dt - \Lambda \lambda^2 \int_{a}^{v}|{^{C}}\mathbb{D}_{a^+}^{\alpha, \sigma} u_0|^2dt.
\end{aligned}
$$
Since $\mu_L < 2$, 
$$
\lim_{\lambda \to \infty} \mathcal{J}(\lambda u_0) = -\infty.
$$
Thus, choosing $\lambda_0$ large enough, we can set $e = \lambda_0 u_0$.
\end{proof}

\noindent 
As a consequence of Theorem \ref{Dife2}, Lemma \ref{Mlm2}, Lemma \ref{Mlm4} and Lemma \ref{Mlm3} we get all conditions of Theorem \ref{MPT}, so there is $u\in \mathbb{H}_{0}^{\alpha, \sigma}(a,b)$ such that 
$$
\mathcal{J}(u) = c = \inf_{g\in \Gamma}\max_{t\in [0,1]} \mathcal{J}(g(t))\quad \mbox{and}\quad \mathcal{J}'(u) = 0.
$$   

\section{Conclusions and Open Questions}
\label{sec:Conc}

The fractional calculus is a mathematical area of a currently strong
research, with numerous applications in physics, engineering and medicine.
 In this paper we go a step further: we prove an explicit fractional Noether's theorem.
 A study of fractional optimal control problems in \cite{Bourdin2018}, provides
 a general formulation of a fractional version
 of Pontryagin's Maximum Principle. Then, with a
 fractional notion of Pontryagin extremal,
 one can try to extend the present results to the more
 general context of the fractional optimal control.

\vspace{1cm}
\noindent
{\bf CONTRIBUTION}\\
All authors contributed equally to this work.

\vspace{.5cm}

\noindent 
{\bf CONFLICT OF INTEREST}\\
This work does not have any conflicts of interest.

\vspace{.5cm}
\noindent 
{\bf DATA SUPPORTING}\\
The manuscript has no associated data. 

\vspace{.5cm}
\noindent
{\bf ACKNOWLEDGEMENTS}\\
The author warmly thanks the anonymous referees for their useful and nice comments on the paper.



\end{document}